\documentclass[reqno,10pt,centertags]{amsart} 
\usepackage{amsmath,amsthm,amscd,amssymb,latexsym,esint,upref,stmaryrd,
enumerate,color,verbatim,yfonts,enumitem}
\usepackage{hyperref} 
\newcommand*{\mailto}[1]{\href{mailto:#1}{\nolinkurl{#1}}}
\newcommand{\arxiv}[1]{\href{http://arxiv.org/abs/#1}{arXiv:#1}}

\setenumerate{label=$($\textit{\roman*}$)$}

\newcommand{\R}{{\mathbb R}}

\newcommand{\bbC}{{\mathbb{C}}}

\newcommand{\bbN}{{\mathbb{N}}}

\newcommand{\bbR}{{\mathbb{R}}}

\newcommand{\cH}{{\mathcal H}}

\newcommand{\cN}{{\mathcal N}}

\newcommand{\cW}{{\mathcal W}}

\renewcommand{\a}{\alpha}
\renewcommand{\b}{\beta}
\newcommand{\g}{\gamma}

\DeclareMathOperator{\supp}{supp}

\DeclareMathOperator{\dom}{dom}

\renewcommand{\ln}{\text{\rm ln}}

\newcommand{\norm}[1]{\lVert#1\rVert}

\newcommand{\no}{\notag}
\newcommand{\lb}{\label}
\newcommand{\f}{\frac}

\newcommand{\ol}{\overline}
\newcommand{\bs}{\backslash}

\newcommand{\wti}{\widetilde}

\newcommand{\oh}{o}
\newcommand{\hatt}{\widehat} 
\newcommand{\dott}{\,\cdot\,}

\renewcommand{\dot}{\overset{\textbf{\Large.}}}

\renewcommand{\dotplus}{\overset{\textbf{\Large.}} +}

\newcommand{\bi}{\bibitem}

\newcommand{\lam}{\lambda}

\newcommand{\al}{\alpha}
\newcommand{\be}{\beta}

\newcommand{\Lr}{{L^2((a,b);rdx)}}

\newcommand{\ACl}{{AC_{loc}((a,b))}}
\newcommand{\Ll}{{L^1_{loc}((a,b);dx)}}

\makeatletter
\def\theequation{\@arabic\c@equation}

\allowdisplaybreaks 
\numberwithin{equation}{section}
\usepackage{tikz}
\newtheorem{theorem}{Theorem}[section]
\newtheorem{proposition}[theorem]{Proposition}
\newtheorem{lemma}[theorem]{Lemma}
\newtheorem{corollary}[theorem]{Corollary}
\newtheorem{definition}[theorem]{Definition}
\newtheorem{hypothesis}[theorem]{Hypothesis}
\newtheorem{example}[theorem]{Example}

\theoremstyle{remark}

\newenvironment{remark}[1][]{\refstepcounter{theorem}\par\medskip\noindent\textit{Remark~$\theexample. #1$} \rmfamily}{{\ }\hfill $\diamond$ \vspace{6pt}}

\begin{document}

\title[Nonnegative extensions of Sturm--Liouville operators]{Nonnegative extensions of Sturm--Liouville operators with an application to problems with symmetric coefficient functions}

\author[C.\ Fischbacher]{Christoph Fischbacher} 
\address[C. Fischbacher]{Department of Mathematics, 
Baylor University, Sid Richardson Bldg, 1410 S.\,4th Street, Waco, TX 76706, USA}
\email{\mailto{c\_fischbacher@baylor.edu}}
\urladdr{\url{https://math.artsandsciences.baylor.edu/person/christoph-fischbacher-phd}}

\author[J.\ Stanfill]{Jonathan Stanfill}
\address[J. Stanfill]{Department of Mathematics, The Ohio State University \\
100 Math Tower, 231 West 18th Avenue, Columbus, OH 43210, USA}
\email{\mailto{stanfill.13@osu.edu}}
\urladdr{\url{https://u.osu.edu/stanfill-13/}}

\date{\today}
\@namedef{subjclassname@2020}{\textup{2020} Mathematics Subject Classification}

\begin{abstract}
The purpose of this paper is to study nonnegative self-adjoint extensions associated with singular Sturm--Liouville expressions with strictly positive minimal operators. We provide a full characterization of all possible nonnegative self-adjoint extensions of the minimal operator in terms of generalized boundary values, as well as a parameterization of all nonnegative extensions when fixing a boundary condition at one endpoint. In addition, we investigate problems where the coefficient functions are symmetric about the midpoint of a finite interval, illustrating how every self-adjoint operator of this form is unitarily equivalent to the direct sum of two self-adjoint operators restricted to half of the interval. We also extend these result to symmetric two interval problems. We then apply our previous results to parameterize all nonnegative extensions of operators with symmetric coefficient functions. We end with an example of an operator with a symmetric Bessel-type potential (i.e., symmetric confining potential) and an application to integral inequalities.
\end{abstract}

\maketitle


\section{Introduction}\lb{s1}

The purpose of this paper is to study nonnegative self-adjoint extensions associated with singular, three-coefficient Sturm--Liouville expressions
\begin{align*}
    \tau = \dfrac{1}{r(x)}\Big[-\dfrac{d}{dx} p(x)  \dfrac{d}{dx} + q(x) \Big] \ \text{ for a.e.~$x\in(a,b) \subseteq \R$},
\end{align*}
(see Hypothesis \ref{h2.1} for details)
with strictly positive minimal operators. 
While Sturm--Liouville operators are enormously well studied, to the best of our knowledge, such a useful parametrization has not been written down explicitly in terms of generalized boundary values before. For recent treatments of Sturm--Liouville theory with encyclopedic references, we refer the reader to Gesztesy, Nichols, and Zinchenko \cite{GNZ24} and the monograph \cite{Ze05} by Zettl.

Our main result is a full characterization of all nonnegative extensions for quasi-regular (limit circle at both endpoints) problems which are bounded from below- see Theorems \ref{Thm:dim2} and \ref{Thm:dim1}. The upshot of these results is that they hold for any choice of principal and nonprincipal solutions and the ingredients for the formulas are quite simple. It is important for context to note that Brown and Evans studied positive extensions via Birman-Kre\u\i n-Vishik-Grubb theory in \cite{BE16}, providing their characterizations in \cite[Thms. 2.1 and 2.2]{BE16} via principal and nonprincipal solutions of $\tau u=0$. Our approach differs in the use of generalized boundary values and Theorem \ref{t3.1} leading to more explicit expressions that allow us to readily compare to the classic characterization of self-adjoint extensions via separated and coupled boundary conditions, which was not done in \cite{BE16}- see Remark \ref{r3.4} for more details. Moreover, our results completely characterize when and how one can order such extensions, a new result with many nuances occurring since not all nonnegative extensions can be compared to one another. We explicitly compare a special case of our characterization to those of \cite{BE16} in Remark \ref{Rem:3.3} $(iv)$ whenever the right endpoint is in the limit point case.

We start with any principal, $u_a$, and nonprincipal, $\hatt u_a$, solutions of $\tau u=0$ at $x=a$ where each are solutions on $(a,b)$ with Wronskian 1, using $u_a,\hatt u_a$ to define the generalized boundary values at $x=a$ (see \eqref{2.20}, \eqref{2.21}). 
Then a distinguished nonprincipal solution, $\hatt v_a$, which is orthogonal to $u_a$ is found. The only explicit values needed to parameterize all nonnegative self-adjoint extensions consists of the generalized function and derivative values of $u_a,\hatt v_a,$ along with the $\Lr$-norms of $u_a,\hatt v_a$.
As an illustration of our main results, we consider the Bessel equation on an arbitrary finite interval and show how to easily implement our general parameterization- see also \cite{GPS24} for a recent study of Bessel-type operators, domain properties, and generalized boundary values.

After characterizing all nonnegative extensions, we also parameterize all nonnegative extensions of a strictly positive minimal operator with a fixed boundary condition at one endpoint in terms of generalized boundary values. While this part of the investigation was originally motivated by specific examples considered in the classic Alonso--Simon paper \cite{AS80}, as we will illustrate, this parameterization is particularly useful when the underlying coefficient functions of the differential expression satisfy certain symmetries (including two interval problems). We point out that such operators with symmetries have been studied previously (typically one-dimensional Schr\"odinger operators with potentials that satisfy symmetries), often with applications toward physical models with point interactions. Our results for two interval problems extend to many of these interactions since they often satisfy the symmetries we study. We refer the interested reader to \cite{ADK98,GM10,Ha78,Ha78a,LL24}.

As a natural application of these results, we investigate problems with coefficient functions which are symmetric about the midpoint of a finite interval, illustrating how every (reflection invariant) self-adjoint operator of this form is unitarily equivalent to the direct sum of two self-adjoint operators restricted to half of the interval. In particular, we show that the self-adjoint extensions of the minimal operator, $T_{min}$, associated with $\tau$, where $p,q,r$ are symmetric about $x=(b+a)/2$, satisfy (see Theorem \ref{t2.10} for the definition of the boundary conditions)
\begin{align}
T_{\a,\a}\ \text{ is unitarily equivalent to }\ T^{(a,(b+a)/2)}_{\alpha,\pi}\oplus T^{(a,(b+a)/2)}_{\alpha,\pi/2}, \label{1.2}\\
T_{0,R'}\ \text{ is unitarily equivalent to }\ T^{(a,(b+a)/2)}_{\alpha,\pi}\oplus T^{(a,(b+a)/2)}_{\alpha',\pi/2},\label{1.3}
\end{align}
in $L^2((a,(b+a)/2),rdx)\oplus L^2((a,(b+a)/2),rdx)$, where $\alpha\in(0,\pi]$ in \eqref{1.2} and $R'\in SL(2,\bbR)$ with $R_{11}=R_{22}$ (i.e., equal diagonal entries) in \eqref{1.3}, with $\a,\a'$ in \eqref{1.3} given in terms of trigonometric functions with arguments coming from the entries of $R'$.
This characterization lends itself to our previous fixed boundary condition analysis, allowing us to easily illustrate the ordering of nonnegative extensions (when this ordering exists) by considering the decompositions above. We further point out that \eqref{1.2} and \eqref{1.3} allow one to calculate (or approximate) the eigenvalues much simpler for these equations as one does not need the solution across the entire interval. This will be illustrated by a symmetric Bessel-type example. We further extend the decompositions \eqref{1.2} and \eqref{1.3} to two interval problems on $(-a,0)\cup(0,a)$ whose coefficient functions are symmetric about $x=0$, an important class of problems in many physical problems as mentioned above.

We end by remarking that a natural application of these results comes in the form of integral inequalities, with the study \cite{GPS21} partially motivating our investigation of Sturm--Liouville operators with symmetric coefficient functions. In particular, we apply our results to the Schr\"odinger expression
\begin{align}
\tau_\g=-\dfrac{d^2}{dx^2}+\big[\gamma^2-(1/4)\big]d^{-2}_{(a,b)}(x),\quad
\gamma\in[0,1),\, x\in(a,b),
\end{align}
where
\begin{equation}
d_{(a,b)}(x)=\begin{cases}
x-a,& x\in(a,(b+a)/2],\\
b-x,& x\in[(b+a)/2,b).
\end{cases}
\end{equation}
This then leads to the following inequality, originally proven in \cite{AW07}:
\begin{align}
\int_a^b  |f'(x)|^2\, dx  \geq \big[(1/4)-\g^2\big]\int_a^b d_{(a,b)}^{-2}(x)|f(x)|^2\, dx
+\frac{4\lambda_{\g,1}^2}{(b-a)^2}\int_a^b & |f(x)|^2\, dx,  \notag \\
\g\in[0,1),\ f  \in H_0^1((a,b)),\ a,b\in\bbR&,
\end{align}
where $\lambda_{j,1}$ is the first positive zero of a certain function (see \eqref{5.15}).
In particular, our results on the decomposition of the Friedrichs extension on the whole interval into the direct sum of two self-adjoint extensions on the half interval (\eqref{1.2} with $\a=\pi$) illustrates why the mixed Dirichlet--Neumann eigenvalue (the constant ${4\lambda_{\g,1}^2/(b-a)^2}$ in the example above) occurs. This eigenvalue will in fact be the optimal constant on the integral of $|f|^2$ for an integral inequality (on the appropriate domain) related to Sturm--Liouville problems with fixed symmetric coefficient functions by considering the infimum of the Friedrichs extension spectrum in \eqref{1.2}.

The structure of the paper is as follows. Section \ref{s2} recalls general singular Weyl--Titchmarsh--Kodaira theory that will be utilized throughout. In Section \ref{s3}, we study nonnegative extensions of a positive minimal operator. We then turn to singular Sturm--Liouville operators with coefficient functions which are symmetric about the midpoint
of a finite interval, and two interval problems, in Section \ref{s4}. We end with the Bessel equation and an example of a Schr\"odinger equation with a symmetric Bessel-type potential
(i.e., symmetric confining potential), along with an application to integral inequalities in Section \ref{s5}. Appendix \ref{App} includes the discussion of symmetries of the fundamental system of solutions and a direct factorization of the characteristic function associated with symmetric coefficient problems.

\section{Singular 
Weyl--Titchmarsh--Kodaira theory}\label{s2}

All results surveyed in this section can be found in \cite{GLN20} and \cite[Ch.~13]{GNZ24} which contain very detailed lists of references to the basics of Weyl--Titchmarsh theory. Here we just mention a few additional and classical sources such as \cite[Sect.~129]{AG81}, \cite[Ch.~6]{BHS20}, \cite{BFL20}, 
\cite[Chs.~8, 9]{CL85}, \cite[Sects.~13.6, 13.9, 13.0]{DS88}, 
\cite[Ch.~III]{JR76}, \cite[Ch.~V]{Na68}, \cite{NZ92}, \cite[Ch.~6]{Pe88}, \cite[Ch.~9]{Te14}, \cite[Sect.~8.3]{We80}, \cite[Ch.~13]{We03}, \cite[Chs.~4, 6--8]{Ze05}.

Throughout our discussion of singular Sturm--Liouville operators we make the following assumptions:

\begin{hypothesis} \label{h2.1}
Let $(a,b) \subseteq \bbR$ and suppose that $p,q,r$ are $($Lebesgue\,$)$ measurable functions on $(a,b)$ 
such that the following items $(i)$--$(iii)$ hold: \\[1mm] 
$(i)$ \hspace*{1.1mm} $r>0$ a.e.~on $(a,b)$, $r\in\Ll$. \\[1mm] 
$(ii)$ \hspace*{.1mm} $p>0$ a.e.~on $(a,b)$, $1/p \in\Ll$. \\[1mm] 
$(iii)$ $q$ is real-valued a.e.~on $(a,b)$, $q\in\Ll$. 
\end{hypothesis}

Given Hypothesis \ref{h2.1}, we now study Sturm--Liouville operators associated with the general, 
three-coefficient differential expression $\tau$ of the type,
\begin{align}\lb{2.1}
\tau=\f{1}{r(x)}\left[-\f{d}{dx}p(x)\f{d}{dx} + q(x)\right] \, \text{ for a.e.~$x\in(a,b) \subseteq \R$.} 
\end{align} 

\begin{definition} 
Assume Hypothesis \ref{h2.1}. Given $\tau$ as in \eqref{2.1}, the \textit{maximal operator} $T_{max}$ in $\Lr$ associated with $\tau$ is defined by
\begin{align}
&T_{max} f = \tau f,    \notag
\\
& f \in \dom(T_{max})=\big\{g\in\Lr \, \big| \,g,g^{[1]}\in\ACl;    \\ 
& \hspace*{6.3cm}  \tau g\in\Lr\big\},  \notag
\end{align}
with the Wronskian $($and quasi-derivative$)$ of $f,g\in\ACl$ defined by
\begin{equation}
W(f,g)(x) = f(x)g^{[1]}(x) - f^{[1]}(x)g(x), \quad
y^{[1]}(x) = p(x) y'(x), \quad x \in (a,b).
\end{equation}
The \textit{preminimal operator} $T_{min,0} $ in $\Lr$ associated with $\tau$ is defined by 
\begin{align}
&T_{min,0}  f = \tau f,   \notag
\\
&f \in \dom (T_{min,0})=\big\{g\in\Lr \, \big| \, g,g^{[1]}\in\ACl;   
\\
&\hspace*{3.25cm} \supp \, (g)\subset(a,b) \text{ is compact; } \tau g\in\Lr\big\}.   \notag
\end{align}
One can prove that $T_{min,0} $ is symmetric and thus closable. One then defines $T_{min}$ 
as the closure of $T_{min,0} $.
\end{definition}

It is well known that
$(T_{min,0})^* = T_{max},$
and hence $T_{max}$ is closed and $T_{min}=\ol{T_{min,0} }$ is given by
\begin{align}
&T_{min} f = \tau f, \no
\\
&f \in \dom(T_{min})=\big\{g\in\Lr  \, \big| \,  g,g^{[1]}\in\ACl;      \\
& \qquad \text{for all } h\in\dom(T_{max}), \, W(h,g)(a)=0=W(h,g)(b); \, \tau g\in\Lr\big\}   .
\no
\end{align}
Moreover, $T_{min,0} $ is essentially self-adjoint if and only if\; $T_{max}$ is symmetric, and then 
$\ol{T_{min,0} }=T_{min}=T_{max}$.

The celebrated Weyl alternative then can be stated as follows:

\begin{theorem}[Weyl's Alternative] \label{t2.3} ${}$ \\
Assume Hypothesis \ref{h2.1}. Then the following alternative holds: \\[1mm] 
$(i)$ For every $z\in\bbC$, all solutions $u$ of $(\tau-z)u=0$ are in $\Lr$ near $a$ 
$($resp., near $b$$)$. \\[1mm] 
$(ii)$  For every $z\in\bbC$, there exists at least one solution $u$ of $(\tau-z)u=0$ which is not in $\Lr$ near $a$ $($resp., near $b$$)$. In this case, for each $z\in\bbC\bs\bbR$, there exists precisely one solution $u_a$ $($resp., $u_b$$)$ of $(\tau-z)u=0$ $($up to constant multiples$)$ which lies in $\Lr$ near $a$ $($resp., near $b$$)$. 
\end{theorem}

This yields the limit circle/limit point classification of $\tau$ at an interval endpoint and links self-adjointness of $T_{min}$ (resp., $T_{max}$) and the limit point property of $\tau$ at both endpoints as follows. 

\begin{definition} Assume Hypothesis \ref{h2.1}. \\[1mm]  
In case $(i)$ in Theorem \ref{t2.3}, $\tau$ is said to be in the \textit{limit circle case} at $a$ $($resp., at $b$$)$. $($Frequently, $\tau$ is then called \textit{quasi-regular} at $a$ $($resp., $b$$)$.$)$
\\[1mm] 
In case $(ii)$ in Theorem \ref{t2.3}, $\tau$ is said to be in the \textit{limit point case} at $a$ $($resp., at $b$$)$. \\[1mm]
If $\tau$ is in the limit circle case at $a$ and $b$ then $\tau$ is also called \textit{quasi-regular} on $(a,b)$. 
\end{definition}

\begin{theorem} 
Assume Hypothesis~\ref{h2.1}, then the following items $(i)$ and $(ii)$ hold: \\[1mm] 
$(i)$ If $\tau$ is in the limit point case at $a$ $($resp., $b$$)$, then 
\begin{equation} 
W(f,g)(a)=0 \, \text{$($resp., $W(f,g)(b)=0$$)$ for all $f,g\in\dom(T_{max})$.} 
\end{equation} 
$(ii)$ Recalling $T_{min}=\ol{T_{min,0} }$, one has
\begin{align}
\begin{split}
n_\pm(T_{min}) &= \dim(\ker(T_{max} \mp i I))    \\
& = \begin{cases}
2 & \text{if $\tau$ is in the limit circle case at $a$ and $b$,}\\
1 & \text{if $\tau$ is in the limit circle case at $a$} \\
& \text{and in the limit point case at $b$, or vice versa,}\\
0 & \text{if $\tau$ is in the limit point case at $a$ and $b$}.
\end{cases}
\end{split} 
\end{align}
In particular, $T_{min} = T_{max}$ is self-adjoint if and only if $\tau$ is in the limit point case at $a$ and $b$. 
\end{theorem}

We now assume that $T_{min}$ is bounded from below so that we can characterize boundary conditions utilizing the notions of principal and nonprincipal solutions, a notion originally due to Leighton and Morse \cite{LM36}, Rellich \cite{Re43}, \cite{Re51}, and Hartman and Wintner \cite[Appendix]{HW55} (see also \cite{CGN16}, \cite[Sects.~13.6, 13.9, 13.0]{DS88}, 
\cite[Ch.~XI]{Ha02}, \cite{NZ92}, \cite[Chs.~4, 6--8]{Ze05}). We begin with the following definition.

\begin{definition}
Assume Hypothesis \ref{h2.1}. \\[1mm] 
$(i)$ Fix $c\in (a,b)$ and $\lambda\in\bbR$. Then $\tau - \lam$ is
called {\it nonoscillatory} at $a$ $($resp., $b$$)$, 
if every real-valued solution $u(\lambda,\dott)$ of 
$\tau u = \lambda u$ has finitely many
zeros in $(a,c)$ $($resp., $(c,b)$$)$. Otherwise, $\tau - \lam$ is called {\it oscillatory}
at $a$ $($resp., $b$$)$. \\[1mm] 
$(ii)$ Let $\lambda_0 \in \bbR$. Then $T_{min}$ is called bounded from below by $\lambda_0$, 
and one writes $T_{min} \geq \lambda_0 I$, if 
\begin{equation} 
(u, [T_{min} - \lambda_0 I]u)_{L^2((a,b);rdx)}\geq 0, \quad u \in \dom(T_{min}).
\end{equation}
\end{definition}

The following is a key result in the construction of generalized boundary values. 

\begin{theorem} \label{t2.7} 
Assume Hypothesis \ref{h2.1}. Then the following items $(i)$--$(iii)$ are
equivalent: \\[1mm] 
$(i)$ $T_{min}$ $($and hence any symmetric extension of $T_{min})$
is bounded from below. \\[1mm] 
$(ii)$ There exists a $\nu_0\in\bbR$ such that for all $\lambda < \nu_0$, $\tau - \lam$ is
nonoscillatory at $a$ and $b$. \\[1mm]
$(iii)$ For fixed $c, d \in (a,b)$, $c \leq d$, there exists a $\nu_0\in\bbR$ such that for all
$\lambda<\nu_0$, $\tau u = \lambda u$ has $($real-valued\,$)$ nonvanishing solutions
$u_a(\lambda,\dott) \neq 0$,
$\hatt u_a(\lambda,\dott) \neq 0$ in $(a,c]$, and $($real-valued\,$)$ nonvanishing solutions
$u_b(\lambda,\dott) \neq 0$, $\hatt u_b(\lambda,\dott) \neq 0$ in $[d,b)$, such that 
\begin{align}
&W(\hatt u_a (\lambda,\dott),u_a (\lambda,\dott)) = 1,
\quad u_a (\lambda,x)=\oh(\hatt u_a (\lambda,x))
\text{ as $x\downarrow a$,} \\
&W(\hatt u_b (\lambda,\dott),u_b (\lambda,\dott))\, = 1,
\quad u_b (\lambda,x)\,=\oh(\hatt u_b (\lambda,x))
\text{ as $x\uparrow b$,} \\
&\int_a^c dx \, p(x)^{-1}u_a(\lambda,x)^{-2}=\int_d^b dx \, 
p(x)^{-1}u_b(\lambda,x)^{-2}=\infty, \\
&\int_a^c dx \, p(x)^{-1}{\hatt u_a(\lambda,x)}^{-2}<\infty, \quad 
\int_d^b dx \, p(x)^{-1}{\hatt u_b(\lambda,x)}^{-2}<\infty.
\end{align}
\end{theorem}

\begin{definition}
Assume Hypothesis \ref{h2.1}, suppose that $T_{min}$ is bounded from below, and let 
$\lambda\in\bbR$. Then $u_a(\lambda,\dott)$ $($resp., $u_b(\lambda,\dott)$$)$ in Theorem
\ref{t2.7}\,$(iii)$ is called a {\it principal} $($or {\it minimal}\,$)$
solution of $\tau u=\lambda u$ at $a$ $($resp., $b$$)$. A real-valued solution of $\tau
u=\lambda u$ linearly independent of $u_a(\lambda,\dott)$ $($resp.,
$u_b(\lambda,\dott)$$)$ is called {\it nonprincipal} at $a$ $($resp., $b$$)$.
\end{definition}

Principal and nonprincipal solutions are well-defined due to Lemma \ref{l2.9}\,$(i)$ below. 

\begin{lemma} \label{l2.9} Assume Hypothesis \ref{h2.1} and suppose that $T_{min}$ is bounded 
from below. \\[1mm]
$(i)$ $u_a(\lambda,\dott)$ and $u_b(\lambda,\dott)$ in Theorem
\ref{t2.7}\,$(iii)$ are unique up to $($nonvanishing\,$)$ real constant multiples. Moreover,
$u_a(\lambda,\dott)$ and $u_b(\lambda,\dott)$ are minimal solutions of
$\tau u=\lambda u$ in the sense that 
\begin{align}
u(\lambda,x)^{-1} u_a(\lambda,x)&=\oh(1) \text{ as $x\downarrow a$,} \\ 
u(\lambda,x)^{-1} u_b(\lambda,x)&=\oh(1) \text{ as $x\uparrow b$,}
\end{align}
for any other solution $u(\lambda,\dott)$ of $\tau u=\lambda u$
$($which is nonvanishing near $a$, resp., $b$$)$ with
$W(u_a(\lambda,\dott),u(\lambda,\dott))\neq 0$, respectively, 
$W(u_b(\lambda,\dott),u(\lambda,\dott))\neq 0$. \\[1mm]
$(ii)$ Let $u(\lambda,\dott) \neq 0$ be any nonvanishing solution of $\tau u=\lambda
u$ near $a$ $($resp., $b$$)$. Then for $c_1>a$ $($resp., $c_2<b$$)$
sufficiently close to $a$ $($resp., $b$$)$, 
\begin{align}
& \hatt u_a(\lambda,x)=u(\lambda,x)\int_x^{c_1}dx' \, 
p(x')^{-1}u(\lambda,x')^{-2} \\
& \quad \bigg(\text{resp., }\hatt u_b(\lambda,x)=u(\lambda,x)\int^x_{c_2}dx' \, 
p(x')^{-1}u(\lambda,x')^{-2}\bigg) 
\end{align} 
is a nonprincipal solution of $\tau u=\lambda u$ at $a$ $($resp.,
$b$$)$. If $\hatt u_a(\lambda,\dott)$ $($resp., $\hatt
u_b(\lambda,\dott))$ is a nonprincipal solution of $\tau u=\lambda u$
at $a$ $($resp., $b$$)$ then 
\begin{align}
& u_a(\lambda,x)=\hatt u_a(\lambda,x)\int_a^{x}dx' \, 
p(x')^{-1}{\hatt u_a(\lambda,x')}^{-2} \\
& \quad \bigg(\text{resp., } u_b(\lambda,x)=\hatt u_b(\lambda,x)\int^b_{x}dx' \, 
p(x')^{-1}{\hatt u_b(\lambda,x')}^{-2}\bigg)
\end{align} 
is a principal solution of $\tau u=\lambda u$ at $a$ $($resp., $b$$)$. 
\end{lemma}

Given these oscillation theoretic preparations, one can now describe all self-adjoint extensions of $T_{min}$, a recent result from \cite{GLN20}.

\begin{theorem}[{\cite[Thm. 4.5]{GLN20}}] \label{t2.10}
Assume Hypothesis \ref{h2.1} and that $\tau$ is in the limit circle case at $a$ and $b$ $($i.e., $\tau$ is quasi-regular 
on $(a,b)$$)$. In addition, assume that $T_{min} \geq \lambda_0 I$ for some $\lambda_0 \in \bbR$, and denote by 
$u_a(\lambda_0, \dott)$ and $\hatt u_a(\lambda_0, \dott)$ $($resp., $u_b(\lambda_0, \dott)$ and 
$\hatt u_b(\lambda_0, \dott)$$)$ principal and nonprincipal solutions of $\tau u = \lambda_0 u$ at $a$ 
$($resp., $b$$)$, satisfying
\begin{equation}
W(\hatt u_a(\lambda_0,\dott), u_a(\lambda_0,\dott)) = W(\hatt u_b(\lambda_0,\dott), u_b(\lambda_0,\dott)) = 1.
\end{equation}  
For $g\in\dom(T_{max})$ we introduce the generalized boundary values
\begin{align}
\begin{split} \label{2.20}
\wti g(a) &= - W(u_a(\lambda_0,\dott), g)(a)  = \lim_{x \downarrow a} \f{g(x)}{\hatt u_a(\lambda_0,x)},    \\
\wti g(b) &=  - W(u_b(\lambda_0,\dott), g)(b)   = \lim_{x \uparrow b} \f{g(x)}{\hatt u_b(\lambda_0,x)},    
\end{split} \\
\begin{split} \label{2.21}
{\wti g}^{\, \prime}(a) &= W(\hatt u_a(\lambda_0,\dott), g)(a)   = \lim_{x \downarrow a} \f{g(x) - \wti g(a) \hatt u_a(\lambda_0,x)}{u_a(\lambda_0,x)},    \\ 
{\wti g}^{\, \prime}(b) &=  W(\hatt u_b(\lambda_0,\dott), g)(b)  = \lim_{x \uparrow b} \f{g(x) - \wti g(b) \hatt u_b(\lambda_0,x)}{u_b(\lambda_0,x)}.   
\end{split} 
\end{align}
Then the following items $(i)$--$(iii)$ hold: \\[1mm] 
$(i)$ All self-adjoint extensions $T_{\al,\be}$ of $T_{min}$ with separated boundary conditions are of the form
\begin{align}
& T_{\al,\be} f = \tau f, \quad \al,\be\in(0,\pi],   \notag\\
& f \in \dom(T_{\al,\be})=\big\{g\in\dom(T_{max}) \, \big| \, \wti g(a)\cos(\al)+ \wti g^{\, \prime}(a)\sin(\al)=0;   \label{2.25} \\ 
& \hspace*{5.5cm} \, \wti g(b)\cos(\be)- \wti g^{\, \prime}(b)\sin(\be) = 0 \big\}.    \notag
\end{align}
Special cases: $\al=\pi$ $($i.e., $\wti g(a)=0$$)$ is called the Dirichlet-type boundary condition at $a$; $\al=\pi/2$,  
$($i.e., $\wti g^{\, \prime}(a)=0$$)$ is called the Neumann-type boundary condition at $a$ $($analogous facts hold at the 
endpoint $b$$)$. \\[1mm]
$(ii)$ All self-adjoint extensions $T_{\eta,R}$ of $T_{min}$ with coupled boundary conditions are of the form
\begin{align}\label{2.26}
\begin{split} 
& T_{\eta,R} f = \tau f,    \\
& f \in \dom(T_{\eta,R})=\bigg\{g\in\dom(T_{max}) \, \bigg| \begin{pmatrix} \wti g(b)\\ \wti g^{\, \prime}(b)\end{pmatrix} 
= e^{i\eta}R \begin{pmatrix}
\wti g(a)\\ \wti g^{\, \prime}(a)\end{pmatrix} \bigg\},
\end{split}
\end{align}
where $\eta\in[0,\pi)$, and $R$ is a real $2\times2$ matrix with $\det(R)=1$ 
$($i.e., $R \in SL(2,\bbR)$$)$.
Special cases: $\eta = 0$, $R=I_2$ $($i.e., $\wti g(b)=\wti g(a)$, $\wti g^{\, \prime}(b)=\wti g^{\, \prime}(a)$$)$ are called {\it periodic boundary conditions}; 
similarly, $\eta = 0$, $R= - I_2$ $($i.e., $\wti g(b)=-\wti g(a)$, $\wti g^{\, \prime}(b)=-\wti g^{\, \prime}(a)$$)$ are called {\it antiperiodic boundary conditions}. 
\\[1mm]  
$(iii)$ Every self-adjoint extension of $T_{min}$ is either of type $(i)$ or of type 
$(ii)$.
\end{theorem}

\begin{remark}
$(i)$ We point out that our choice of identifying the Dirichlet-type boundary condition at $a$ with the choice $\alpha=\pi$ rather than the typical choice $0$ stems from our characterization of nonnegative extensions in Section \ref{s3}. In fact, this seems to be the natural choice by our results. \\[1mm]
$(ii)$ If $\tau$ is in the limit point case at one endpoint, say, at the endpoint $b$, one omits the corresponding boundary condition involving $\b \in (0, \pi]$ at $b$ in \eqref{2.25} to obtain all self-adjoint extensions $T_{\a}$ of 
$T_{min}$, indexed by $\a \in (0, \pi]$. In the case where $\tau$ is in the limit point case at both endpoints, all boundary values and boundary conditions become superfluous as in this case $T_{min} = T_{max}$ is self-adjoint. \\[1mm] 
$(iii)$ An explicit calculation demonstrates that for $g, h \in \dom(T_{max})$,
\begin{equation}
\wti g(d) \wti h^{\, \prime}(d) - \wti g^{\, \prime}(d) \wti h(d) = W(g,h)(d), \quad d \in \{a,b\},   \label{2.27}
\end{equation} 
interpreted in the sense that either side in \eqref{2.27} has a finite limit as $d \downarrow a$ and $d \uparrow b$. 
Of course, for \eqref{2.27} to hold at $d \in \{a,b\}$, it suffices that $g$ and $h$ lie locally in $\dom(T_{max})$ near $x=d$.\\[1mm]
$(iv)$ While the principal solution at an endpoint is unique up to constant multiples (which we will ignore), nonprincipal solutions differ by additive constant multiples of the principal solution. As a result, if 
\begin{align}
\begin{split} 
& \hatt u_a (\lambda_0, \dott) \longrightarrow \hatt u_a (\lambda_0, \dott) + C u_a(\lambda_0,\dott), \quad C \in \bbR,    \\
& \quad \text{then } \, \wti g(a) \longrightarrow \wti g(a), 
\quad {\wti g}^{\, \prime}(a) \longrightarrow {\wti g}^{\, \prime}(a) - C \wti g(a),    
\end{split} 
\end{align}
and analogously at the endpoint $b$, facts that will be revisited in Remark \ref{Rem:3.3} $(ii)$. Hence, generalized boundary values ${\wti g}^{\, \prime}(d)$ at the endpoint 
$d \in \{a,b\}$ depend on the choice of nonprincipal solution $\hatt u_{d}(\lambda_0, \dott)$ of $\tau u = \lambda_0 u$ at $d$. However, the Friedrichs boundary conditions $\wti g(a) = 0 = \wti g(b)$ are independent of the choice of 
nonprincipal solution (see \eqref{2.30}). Furthermore, if the problem considered is regular (i.e., the integrability assumptions in Hypothesis \ref{h2.1} extend to the endpoints and the interval is finite), the principal and nonprincipal solutions can be chosen to agree exactly with the regular boundary conditions.
\end{remark} 

As a special case, we recall that the domain of $T_{min}$ is now characterized by 
\begin{equation}
\dom(T_{min})= \big\{g\in\dom(T_{max})  \, \big| \, \wti g(a) = {\wti g}^{\, \prime}(a) = \wti g(b) = {\wti g}^{\, \prime}(b) = 0\big\}.
\end{equation}
Similarly, the Friedrichs extension $T_F$ of $T_{min}$ now permits a particularly simple characterization in terms of the generalized boundary values $\wti g(a), \wti g(b)$ as derived by Kalf  \cite{Ka78} and subsequently by Niessen and Zettl \cite{NZ92} (see also \cite{Re51}, \cite{Ro85} and the extensive literature cited in \cite{GLN20}, 
\cite[Ch.~13]{GNZ24}):
\begin{align}\label{2.30}
T_F f = \tau f, \quad f \in \dom(T_F)= \big\{g\in\dom(T_{max})  \, \big| \, \wti g(a) = \wti g(b) = 0\big\}.
\end{align}
Finally, we also mention the special case of the Krein--von Neumann extension in the case where 
$T_{min}$ is strictly positive, a result recently derived in \cite{FGKLNS21}.

\begin{theorem}[{\cite[Thm. 3.5]{FGKLNS21}}]\label{t2.12}
In addition to Hypothesis \ref{h2.1}, suppose $T_{min} \geq \varepsilon I$ for some 
$\varepsilon > 0$. Then the following items $(i)$ and $(ii)$ hold: \\[1mm]
$(i)$ Assume that $n_{\pm}(T_{min}) = 1$ and denote the principal solutions of $\tau u = 0$ at $a$ and $b$ by $u_a(0,\dott)$ and $u_b(0,\dott)$, respectively. If $\tau$ is in the limit circle case at $a$ and in the limit point case at $b$, then the Krein--von Neumann extension $T_{\a_K}$ of $T_{min}$ is given by 
\begin{align}\label{2.31}
& T_{\a_K}f = \tau f,   \no \\
& f \in \dom(T_{\a_K})=\big\{g\in\dom(T_{max}) \, \big| \, \sin(\a_K) {\wti g}^{\, \prime}(a) 
+ \cos(\a_K) \wti g(a) = 0\big\},     \\
& \cot(\a_K) = - \wti u_b^{\, \prime}(0,a) / \wti u_b(0,a), \quad \a_K \in (0,\pi).   \no 
\end{align}
Similarly, if $\tau$ is in the limit circle case at $b$ and in the limit point case at $a$, then the Krein--von Neumann extension $T_{\b_K}$ of $T_{min}$ is given by 
\begin{align}
& T_{\b_K}f = \tau f,   \no \\
& f \in \dom(T_{\b_K})=\big\{g\in\dom(T_{max}) \, \big| \, \sin(\b_K) {\wti g}^{\, \prime}(b) 
- \cos(\b_K) \wti g(b) = 0\big\},     \\
& \cot(\b_K) = \wti u_a^{\, \prime}(0,b) / \wti u_a(0,b), \quad \b_K \in (0,\pi).   \no 
\end{align}
$(ii)$ Assume that $n_{\pm}(T_{min}) = 2$, that is, $\tau$ is in the limit circle case at $a$ and $b$. Then
the Krein--von Neumann extension $T_{0,R_K}$ of $T_{min}$ is given by 
\begin{align}
\begin{split} 
& T_{0,R_K} f = \tau f,    \\
& f \in \dom(T_{0,R_K})=\bigg\{g\in\dom(T_{max}) \, \bigg| \begin{pmatrix} \wti g(b) 
\\ {\wti g}^{\, \prime}(b) \end{pmatrix} = R_K \begin{pmatrix}
\wti g(a) \\ {\wti g}^{\, \prime}(a) \end{pmatrix} \bigg\},     
\end{split}   
\end{align}
where 
\begin{align}\label{2.31a}
R_K=\begin{pmatrix} \wti{\hatt u}_{a}(0,b) & \wti u_{a}(0,b)\\
\wti{\hatt u}_{a}^{\, \prime}(0,b) & \wti u_{a}^{\, \prime}(0,b)
\end{pmatrix}. 
\end{align}
\end{theorem}

\section{Nonnegative extensions}\label{s3}
We now study nonnegative self-adjoint extensions of $T_{min}$ whenever $T_{min}\geq \varepsilon I$ for some $\varepsilon>0$. We will find a description of all nonnegative self-adjoint realizations of $\tau$ in terms of generalized boundary conditions. To this end, let us introduce the following partial order on nonnegative self-adjoint operators:
\begin{definition}
Let $A$ and $B$ be two nonnegative self-adjoint operators in the Hilbert spaces $\mathcal{W}_A$ and $\mathcal{W}_B$, respectively. Assume that there is a Hilbert space $\mathcal{H}$ such that $\mathcal{W}_A, \mathcal{W}_B\subseteq\mathcal{H}$. 
Then the partial order $A\leq B$ is defined as
\begin{equation}
A\leq B\: \Leftrightarrow \: \dom(A^{1/2})\supseteq \dom(B^{1/2})\:\mbox{ and }\: \|A^{1/2}f\|_\cH\leq\|B^{1/2}f\|_\cH\quad\forall f\in\dom(B^{1/2}).
\end{equation}
Here, $A^{1/2}$ and $B^{1/2}$ denote the unique nonnegative self-adjoint square-roots of $A$ in $\mathcal{W}_A$ and $B$ in $\mathcal{W}_B$, respectively.
\end{definition}
\begin{remark}
The reason we state this definition using the Hilbert spaces $\mathcal{W}_A$ and $\mathcal{W}_B$ is that it  allows us to define this partial order also in the case that the operator $B$ is only defined in a subspace of $\overline{\dom(A)}$.
\end{remark}

It is a celebrated result of M.\ G.\  Krein \cite{K47} that there exists two distinguished nonnegative self-adjoint extensions, $T_F$ and $T_K$ (the extensions previously referred to as the ``Friedrichs" and ``Krein-von Neumann" extension, respectively) such that every other nonnegative self-adjoint extension $\hatt{T}$ of $T$ satisfies $T_K\leq \hatt T\leq T_F$. For the strictly positive case $T_{min}\geq \varepsilon I$, all other nonnegative self-adjoint extensions $\hatt{T}$ are parametrized by closed subspaces $\mathcal{W}\subseteq\ker(T_{max})$ and nonnegative auxiliary operators $B$ in $\mathcal{W}$ \cite{B56, F75, V52}. See also \cite{G68} and the more recent \cite{F19}, where sectorial extensions of $T_{min}$ are parametrized using sectorial auxiliary operators $B$.

We now recall the following general result:

\begin{theorem}[{\cite[Prop.1.1]{F17}, \cite[Thm. D.3.13]{GNZ24}}]\lb{t3.1}
Suppose that $T$ is a densely defined, symmetric, closed operator
with nonzero deficiency indices in $\cH$ that satisfies
$T \geq \varepsilon I_{\cH}$ for some $\varepsilon >0$.
Then there exists a one-to-one correspondence between nonnegative self-adjoint operators
$0 \leq B:\dom(B)\subseteq \cW\to \cW$, $\ol{\dom(B)}=\cW$, where $\cW$
is a closed subspace of $\cN_0 :=\ker(T^*)$, and nonnegative self-adjoint
extensions $T_{B,\cW}\geq 0$ of $T$. More specifically, $T_F$ is invertible,
$T_F\geq \varepsilon I_{\cH}$,
and
\begin{align}
& \dom(T_{B,\cW}) = \no\\&\qquad\qquad\dom(T)\dot{+}\big\{((T_F)^{-1}B+I)\eta \,|\,\eta\in\dom(B)\big\} \dot{+}\big\{(T_F)^{-1}v\,|\, v \in \cN_0 \cap \cW^{\bot}\big\}\no\\
& \qquad T_{B,\cW} = T^*|_{\dom(T_{B,\cW})},       \lb{3.1}
\end{align}
where $\cW^{\bot}$ denotes the orthogonal complement of $\cW$ in $\cH$. In
addition,
\begin{align}
&\dom\big((T_{B,\cW})^{1/2}\big) = \dom\big((T_F)^{1/2}\big) \dotplus
\dom\big(B^{1/2}\big),   \\
&\big\|(T_{B,\cW})^{1/2}(u+g)\big\|_{\cH}^2 =\big\|(T_F)^{1/2} u\big\|_{\cH}^2
+ \big\|B^{1/2} g\big\|_{\cH}^2, \\
& \hspace*{2.3cm} u \in \dom\big((T_F)^{1/2}\big), \; g \in \dom\big(B^{1/2}\big),  \no
\end{align}
implying $\ker(T_{B,\cW})=\ker(B).$
Moreover,
\begin{equation}
B \leq \wti B \, \text{ implies } \, T_{B,\cW} \leq T_{\wti B,\wti \cW},
\end{equation}
where
\begin{align}
\begin{split}
& B\colon \dom(B) \subseteq \cW \to \cW, \quad
 \wti B\colon \dom\big(\wti B\big) \subseteq \wti \cW \to \wti \cW,   \\
& \ol{\dom\big(\wti B\big)} = \wti\cW \subseteq \cW = \ol{\dom(B)}.
\end{split}
\end{align}

In the above scheme, the Krein--von Neumann extension $T_K$
of $T$ corresponds to the choice $\cW=\cN_0$ and $B=0$ $($with
$\dom(B)=\dom\big(B^{1/2}\big)=\cN_0=\ker (T^*)$$)$, and the Friedrichs extension $T_F$ corresponds to $\dom(B)=\{0\}$ $($i.e., formally,
$B = \infty$$)$.
\end{theorem}

\subsection{General quasi-regular case}\label{subgen}
We begin our analysis with the nonoscillatory, quasi-regular setting so that boundary conditions are needed at both endpoints (i.e., $T_{min}$ has deficiency index 2).
Start by considering any principal, $u_a$, and nonprincipal, $\hatt u_a$, solutions of $\tau u=0$ at $x=a$ where each are solutions on $(a,b)$ with Wronskian 1, using $u_a,\hatt u_a$ to define the generalized boundary values at $x=a$ in \eqref{2.20}, \eqref{2.21}. 
We now distinguish a specific nonprincipal solution, $\hatt v_a$, by
\begin{equation}\label{B.3.8}
\hatt v_a(0,x):=\hatt u_a(0,x)-\frac{\langle\hatt u_a,u_a\rangle_{L^2_r}}{\norm{u_a}_{L^2_r}^2}u_a(0,x),\quad x\in(a,b),
\end{equation}
where we abbreviate $\Lr$ by $L_r^2$ in the norm and inner product.
Note that $\hatt v_a$ satisfies the same leading behavior as $\hatt u_a$ near $x=a$, is orthogonal to $u_a$ by construction (which simplifies the following analysis), and
\begin{equation} \label{eq:3.9}
\wti{\hatt v}_a(0,a)=1=\wti u_a^{\, \prime}(0,a),\quad \wti{u}_a(0,a)=0,\quad \wti{\hatt v}_a^{\,\prime}(0,a)=-\frac{\langle\hatt u_a,u_a\rangle_{L^2_r}}{\norm{u_a}_{L^2_r}}.
\end{equation}

Using the variation of constants formula now yields the general form for a solution of $\tau \xi=f$ (for an appropriate $f$ to be specified shortly) such that $\wti\xi(a)=0=\wti\xi(b)$,
\begin{align}
\xi(x)&=\hatt v_a(0,x)\int_a^x u_a(0,t)f(t)r(t)dt \label{B.1} \\
&\quad + u_a(0,x)\bigg(\int_x^b \hatt v_a(0,t)f(t)r(t)dt-\frac{\wti{\hatt v}_a(0,b)}{\wti{u}_a(0,b)}\int_a^b u_a(0,t)f(t)r(t)dt\bigg). \no
\end{align}
Note that $\wti u_a(0,b)\neq0$, since if it were then zero would be an eigenvalue of the Friedrichs extension, contradicting our strict positivity assumptions.
From \eqref{B.1}, it is easy to verify the following (implementing $W(\hatt v_a,u_a)=1$):
\begin{align}
&\wti\xi(a)=0=\wti\xi(b),\\
&\wti{\xi}^{\, \prime}(a)=\int_a^b\bigg[\hatt v_a(0,t)-\frac{\wti{\hatt v}_a(0,b)}{\wti{u}_a(0,b)}u_a(0,t)\bigg]f(t)r(t)dt,\label{B.3}\\
&\wti{\xi}^{\, \prime}(b)=-\frac{1}{\wti{u}_a(0,b)}\int_a^b u_a(0,t)f(t)r(t)dt.\label{B.4}
\end{align}

We now specialize to when $f$ equals $u_a$ or $\hatt v_a$ (see \eqref{B.3.8}) by defining
\begin{equation}\label{B.5}
\xi_a(x):=(T_F)^{-1}(u_a)(x),\quad \hatt\xi_a(x):=(T_F)^{-1}(\hatt v_a)(x).
\end{equation}
From \eqref{B.3} and \eqref{B.4}, and employing $W(\hatt v_a,u_a)=1$, one concludes that
\begin{equation}\label{B.6}
\wti\xi_a^{\, \prime}(a)=-\frac{\wti{\hatt v}_a(0,b)}{\wti{u}_a(0,b)}\norm{u_a}_{L^2_r}^2,\quad \wti\xi_a^{\, \prime}(b)=-\frac{\norm{u_a}_{L^2_r}^2}{\wti{u}_a(0,b)},
\end{equation}
and
\begin{equation}\label{B.7}
\wti{\hatt \xi}_a^{\, \prime}(a)=\norm{\hatt v_a}_{L_r^2}^2,\quad \wti{\hatt \xi}_a^{\, \prime}(b)=0.
\end{equation}

The functions $\xi_a$ and $\hatt \xi_a$ given in \eqref{B.5}, along with their boundary values \eqref{B.6} and \eqref{B.7}, are the main ingredients needed in the present investigation.

\subsubsection{The case \texorpdfstring{$\dim(\mathcal{W})=2$}{dim2}}
Using Theorem \ref{t3.1}, our goal is to give a description of all nonnegative self-adjoint extensions of $T_{min}$ in terms of generalized boundary conditions. We begin with the case $\dim(\mathcal{W})=2$ so that $\mathcal{W}=\mathcal{N}_0=\mbox{span}\{\hatt v_a, u_a\}$. For this case, all nonnegative self-adjoint extensions are parametrized by nonnegative self-adjoint operators $B$ on $\mathcal{N}_0$. Letting
\begin{align}
Bu_a&=b_{11}u_a+b_{12}\hatt v_a,\\
B\hatt v_a&=b_{21}u_a+b_{22}\hatt v_a,
\end{align}
the requirement that $B$ be self-adjoint yields the condition 
\begin{equation}
b_{21}=\overline b_{12}\frac{\|\hatt v_a\|_{L_r^2}^2}{\| u_a\|_{L_r^2}^2},
\end{equation}
and the nonnegativity of $B$ is equivalent to the conditions
\begin{equation}\label{3.20a}
b_{11}, b_{22}\geq 0\quad\mbox{and}\quad b_{11}b_{22}\|u_a\|_{L_r^2}^2-|b_{12}|^2\|\hatt v_a\|^2_{L_r^2}\geq 0.
\end{equation}
In addition to $b_{11}, b_{12}$, and $b_{22}$ satisfying these conditions, let $\hatt{B}$ be another self-adjoint and nonnegative auxiliary operator on $\mathcal{N}_0$. 
It immediately follows that $T_{\hatt{B},\mathcal{N}_0}\geq T_{B,\mathcal{N}_0}\geq 0$ if and only if  $\hatt b_{11}\geq b_{11}$, $\hatt b_{22}\geq b_{22}$, and
\begin{equation}
(\hatt b_{11}-b_{11})(\hatt b_{22}-b_{22})\norm{u_a}_{L_r^2}^2-|\hatt b_{12}-b_{12}|^2\norm{\hatt v_a}_{L_r^2}^2\geq0.
\end{equation}
Now, let us focus on finding the generalized boundary conditions determining $\dom(T_{B,\mathcal{N}_0})$. By Theorem \ref{t3.1}, we have
\begin{align}
&\dom(T_{B,\mathcal{N}_0})=\dom(T_{min})\dot{+}\{((T_F)^{-1}B+I)\eta\:|\:\eta\in\mathcal{N}_0\}\\
&=\dom(T_{min})\dot{+}\mbox{span}\left\{ b_{11}\xi_a+b_{12}\hatt \xi_a+u_a\right\}\dot{+} \mbox{span}\left\{ \overline b_{12}\frac{\|\hatt v_a\|_{L_r^2}^2}{\| u_a\|_{L_r^2}^2}\xi_a+b_{22}\hatt \xi_a+\hatt v_a\right\}. \notag
\end{align}
Using the generalized boundary values of $u_a, \hatt v_a, \xi_a$, and $\hatt \xi_a$ given in equations \eqref{eq:3.9}, \eqref{B.6}, and \eqref{B.7} yields the following description of nonnegative self-adjoint extensions of $T_{min}$ in terms of generalized boundary conditions (assuming \eqref{3.20a}):
\begin{align}
\wti{g}^{\, \prime}(a)&=\frac{\wti g (b)-\wti g(a)\wti{\hatt v}_a(0,b)}{\wti u_a(0,b)}\bigg[1- b_{11}\frac{\wti{\hatt v}_a(0,b)}{\wti u_a(0,b)}\norm{u_a}_{L_r^2}^2+b_{12}\norm{\hatt v_a}_{L_r^2}^2\bigg] \notag\\
&\quad +\wti g(a)\norm{\hatt v_a}_{L_r^2}^2\bigg[b_{22}-\overline{b}_{12}\frac{\wti{\hatt v}_a(0,b)}{\wti u_a(0,b)}+\frac{\wti{\hatt v}_a^{\,\prime}(0,a)}{\norm{\hatt v_a}_{L_r^2}^2}\bigg],  \label{3.24}\\
\wti{g}^{\, \prime}(b)&=\frac{\wti g (b)-\wti g(a)\wti{\hatt v}_a(0,b)}{\wti u_a(0,b)}\bigg[\wti u_a^{\, \prime}(0,b)-\frac{b_{11}\norm{u_a}_{L_r^2}^2}{\wti u_a(0,b)}\bigg]  +\wti g(a)\bigg[\wti{\hatt v}_a^{\, \prime}(0,b)-\frac{\overline{b}_{12}\norm{\hatt v_a}_{L_r^2}^2}{\wti u_a(0,b)}\bigg]. \label{3.24a}
\end{align}

\begin{remark}\label{r3.4}
To compare this description to the previous results of Brown and Evans one can add \eqref{3.24} and \eqref{3.24a}, and then compare the resulting equation to \cite[Eq. (2.13)]{BE16} utilizing \eqref{2.20}, \eqref{2.21}. We point out that this is a straightforward, though somewhat tedious, calculation and thus do not include the explicit details here. The important takeaway is that the description given via \eqref{3.24} and \eqref{3.24a}  allows us to readily write these as classical separated and coupled boundary condition descriptions as well as order the extensions (when possible), as we will show next.

A similar comparison holds when considering our results for the case $\dim(\mathcal{W})=1$ in the next section compared to \cite[Eq. (2.12)]{BE16}.

See Remark \ref{Rem:3.3} $(iv)$ for a explicit comparison of our results to those of \cite{BE16} whenever the right endpoint is in the limit point case.
\end{remark}

Now, to compare these requirements to the separated and coupled boundary conditions in \eqref{2.25} and \eqref{2.26}, we begin by noting that separated boundary conditions occur in \eqref{3.24} if and only if
\begin{equation}\label{3.25a}
1- b_{11}\frac{\wti{\hatt v}_a(0,b)}{\wti u_a(0,b)}\norm{u_a}_{L_r^2}^2+b_{12}\norm{\hatt v_a}_{L_r^2}^2=0.
\end{equation}
Note this implies $b_{12}\in\bbR$. Furthermore, a careful study of the coefficient on $\wti g(a)$ in \eqref{3.24a} shows it is actually just $-1/\wti u_a(0,b)$ times \eqref{3.25a} with $\overline{b}_{12}$ replacing $b_{12}$. Since $b_{12}$ is real, \eqref{3.24a} also becomes separated boundary conditions at $x=b$ if \eqref{3.25a} holds, as it should. Comparing to \eqref{2.25} yields the explicit choices of $\alpha$ and $\beta$, which are given in Theorem \ref{Thm:dim2} below (after solving for $b_{12}$ in \eqref{3.25a}).

To analyze the coupled boundary conditions, we assume \eqref{3.25a} does not hold so that we can solve for $\wti g(b)$ in \eqref{3.24}, then substitute this into \eqref{3.24a} to have each given in the form of \eqref{2.26}. For brevity and convenience, we express this in the next theorem as equations rather than writing what each entry in $R$ is equal to as this simply involves the modulus (with a negative sign if the value is a negative real number or lies in the lower half of the complex plane).

We summarize these results in the following theorem.

\begin{theorem}\label{Thm:dim2}
Assume Hypothesis \ref{h2.1}, that $T_{min}$ is strictly positive, and that $\tau$ is in the limit circle case at $a$ and $b$. In addition, let $b_{11},b_{22}\geq 0,\ b_{12}\in\bbC,$ such that
\begin{equation} \label{eq:3.17}
b_{11}b_{22}\norm{u_a}_{L_r^2}^2-|b_{12}|^2\norm{\hatt v_a}_{L_r^2}^2\geq0.
\end{equation}
If $\dim(\cW)=2$, then the separated boundary conditions for $g\in\dom(T_{max})$ describing nonnegative self-adjoint extensions of $T_{min}$ are given by
\begin{align}
\cot(\a)&=b_{11}\bigg(\frac{\wti{\hatt v}_a(0,b)}{\wti u_a(0,b)}\bigg)^2 \norm{u_a}_{L_r^2}^2-b_{22}\norm{\hatt v_a}_{L_r^2}^2-\wti{\hatt v}_a^{\,\prime}(0,a)-\frac{\wti{\hatt v}_a(0,b)}{\wti u_a(0,b)},\ \a\in(0,\pi),  \label{3.27}\\
\cot(\b)&=\frac{1}{\wti u_a(0,b)}\bigg(\wti u_a^{\, \prime}(0,b)-b_{11}\frac{\norm{u_a}_{L_r^2}^2}{\wti u_a(0,b)}\bigg),\ \b\in(0,\pi),\label{3.28}
\end{align}
whenever
\begin{equation}\label{3.29aa}
b_{12}=\norm{\hatt v_a}_{L_r^2}^{-2}\bigg( b_{11}\frac{\wti{\hatt v}_a(0,b)}{\wti u_a(0,b)}\norm{u_a}_{L_r^2}^2-1\bigg)\quad (\textrm{implying }\ b_{12}\in\bbR),
\end{equation}
and, otherwise, the coupled boundary conditions are given by
\begin{align}
\wti g(b)&=\wti{g}^{\, \prime}(a)\wti u_a(0,b)\bigg[1- b_{11}\frac{\wti{\hatt v}_a(0,b)}{\wti u_a(0,b)}\norm{u_a}_{L_r^2}^2+b_{12}\norm{\hatt v_a}_{L_r^2}^2\bigg]^{-1} \label{3.29a}\\
& \quad+\wti g(a)\Bigg(\wti{\hatt v}_a(0,b)-\norm{\hatt v_a}_{L_r^2}^2\bigg[b_{22}\wti u_a(0,b)-\overline{b}_{12}\wti{\hatt v}_a(0,b)+\wti u_a(0,b)\frac{\wti{\hatt v}_a^{\,\prime}(0,a)}{\norm{\hatt v_a}_{L_r^2}^2}\bigg] \notag \\
&\hspace{4.3cm}\times \bigg[1- b_{11}\frac{\wti{\hatt v}_a(0,b)}{\wti u_a(0,b)}\norm{u_a}_{L_r^2}^2+b_{12}\norm{\hatt v_a}_{L_r^2}^2\bigg]^{-1}\Bigg),\notag \\
\wti{g}^{\, \prime}(b)&=\wti{g}^{\, \prime}(a)\bigg[\wti u_a^{\, \prime}(0,b)-\frac{b_{11}\norm{u_a}_{L_r^2}^2}{\wti u_a(0,b)}\bigg]\bigg[1- b_{11}\frac{\wti{\hatt v}_a(0,b)}{\wti u_a(0,b)}\norm{u_a}_{L_r^2}^2+b_{12}\norm{\hatt v_a}_{L_r^2}^2\bigg]^{-1}  \notag\\
&\quad +\wti g(a)\Bigg\{\wti{\hatt v}_a^{\, \prime}(0,b)-\frac{\overline{b}_{12}\norm{\hatt v_a}_{L_r^2}^2}{\wti u_a(0,b)}-\frac{1}{\wti u_a(0,b)}\bigg[\wti u_a^{\, \prime}(0,b)-\frac{b_{11}\norm{u_a}_{L_r^2}^2}{\wti u_a(0,b)}\bigg] \label{3.30a}\\
& \hspace{1.8cm}\times \norm{\hatt v_a}_{L_r^2}^2\bigg[b_{22}\wti u_a(0,b)-\overline{b}_{12}\wti{\hatt v}_a(0,b)+\wti u_a(0,b)\frac{\wti{\hatt v}_a^{\,\prime}(0,a)}{\norm{\hatt v_a}_{L_r^2}^2}\bigg] \notag \\
&\hspace{3.5cm}\times \bigg[1- b_{11}\frac{\wti{\hatt v}_a(0,b)}{\wti u_a(0,b)}\norm{u_a}_{L_r^2}^2+b_{12}\norm{\hatt v_a}_{L_r^2}^2\bigg]^{-1}\Bigg\}. \notag
\end{align}
Moreover, $0\leq T_{B,\mathcal{N}_0}\leq T_{\hatt B,\mathcal{N}_0}$ if and only if $\hatt b_{11}\geq b_{11}$, $\hatt b_{22}\geq b_{22}$, and
\begin{equation} \label{eq:3.18}
(\hatt b_{11}-b_{11})(\hatt b_{22}-b_{22})\norm{u_a}_{L_r^2}^2-|\hatt b_{12}-b_{12}|^2\norm{\hatt v_a}_{L_r^2}^2\geq0.
\end{equation}
\end{theorem}

\begin{remark}\label{r3.5}
$(i)$ One readily verifies that the choices for $R$ defined through \eqref{3.29a} and \eqref{3.30a} yield $\det(R)=1$.\\[1mm] 
$(ii)$ The choice $B=0$ (i.e., $b_{11}=b_{22}=b_{12}=0$) in Theorem \ref{Thm:dim2} recovers the characterization of the Krein--von Neumann extensions given in \eqref{2.31a} as expected.\\[1mm]
$(iii)$ If one is interested in strictly positive extensions, utilizing the small-$z$ expansion of the characteristic functions given in \eqref{2.43} and \eqref{2.44} readily yields which of these extensions have zero as an eigenvalue.\\[1mm]
$(iv)$ Finally, let us make a comment about when restricting to only considering extensions with separated boundary conditions in Theorem \ref{Thm:dim2}. Given two nonnegative extensions, $T_{\a,\b}$ and $T_{\hatt \a,\hatt\b}$, is it true that $\a\leq \hatt\a$ and $\b\leq\hatt\b$ if and only if $T_{\a,\b}\leq T_{\hatt\a,\hatt\b}$?

Note by \eqref{3.28}, $\hatt{b}_{11}\geq b_{11}$ if and only if $\b\leq \hatt\b$.
By \eqref{3.27}, $\a\leq\hatt\a$ if and only if
\begin{equation}\label{3.33}
(\hatt b_{22}-b_{22})\norm{\hatt v_a}_{L_r^2}^2\geq (\hatt b_{11}-b_{11})\left(\frac{\wti{\hatt v}_a(0,b)}{\wti u_a(0,b)}\right)^2\norm{u_a}_{L_r^2}^2,
\end{equation}
which, in particular, implies that $\hatt b_{22}\geq b_{22}$. However, just because $\hatt b_{22}\geq b_{22}$ it does not in general imply \eqref{3.33} holds.

To consider \eqref{eq:3.18}, we note the difference of \eqref{3.29aa} for $b_{12}$ and $\hatt b_{12}$ yields
\begin{equation}\label{3.34}
(\hatt b_{12}-b_{12})^2=(\hatt b_{11}-b_{11})^2 \left(\frac{\wti{\hatt v}_a(0,b)}{\wti u_a(0,b)}\right)^2\frac{\norm{u_a}_{L_r^2}^4}{\norm{\hatt v_a}_{L_r^2}^4}.
\end{equation}
Substituting \eqref{3.34} into the left-hand side of \eqref{eq:3.18} yields that \eqref{eq:3.18} holds if and only if \eqref{3.33} holds. This finishes the other direction of the implication, showing $\a\leq \hatt\a$ and $\b\leq\hatt\b$ if and only if $T_{\a,\b}\leq T_{\hatt\a,\hatt\b}$ in this setting.

Therefore, when restricting to the separated boundary condition case only in Theorem \ref{Thm:dim2}, it suffices to simply  compare the corresponding angles parameterizing the boundary conditions at each endpoint. Moreover, one cannot compare separated boundary conditions whenever $\a<\a'$ and $\b>\b'$ (or vice versa). In fact, by Theorem \ref{Thm:dim1}, and the fact that the Friedrichs extensions is always larger than all extensions, this is true for all self-adjoint extensions with separated boundary conditions which are nonnegative (regardless of the $\dim(\cW)$).

We will study whenever one endpoint has the same fixed boundary condition in both extensions, say $\b$, in Section \ref{subnonneg}, showing a convenient characterization of the range of admissible choices of $\a$ for nonnegative extensions.
\end{remark}

\subsubsection{The case \texorpdfstring{$\dim(\mathcal{W})=1$}{dim1}}
Next, we study the case $\dim(\mathcal{W})=1$. In this case, letting $c\in\mathbb{C}\cup\{\infty\}$, we parameterize $\mathcal{W}=\mathcal{W}_c$ as follows:
\begin{equation}
\mathcal{W}_c=\begin{cases} \mbox{span}\{\hatt{v}_a+cu_a\},\quad&\mbox{if}\:c\in\mathbb{C},\\ \mbox{span}\{u_a\},\quad&\mbox{if}\:c=\infty.\end{cases}
\end{equation}
It follows that $\mathcal{N}_0\cap\mathcal{W}_c^\perp$ is given by
\begin{equation}
\mathcal{N}_0\cap\mathcal{W}_c^\perp=\begin{cases} \mbox{span}\left\{\overline{c}\frac{\hatt{v}_a}{\|\hatt v_a\|^2_{L_r^2}}-\frac{u_a}{\|u_a\|^2_{L_r^2}}\right\},\quad&\mbox{if}\:c\in\mathbb{C},\\ \mbox{span}\{\hatt v_a\},\quad&\mbox{if}\:c=\infty.\end{cases}
\end{equation}
Any nonnegative auxiliary operator $B$ on $\mathcal{W}_c$ is just a multiplication by a nonnegative scalar $\kappa\geq 0$. For convenience, we introduce the notation $T_{B,\mathcal{W}_c}=:T_{\kappa,c}$. By Theorem \ref{t3.1}, we have $0\leq T_{\kappa,c}\leq T_{\hatt{\kappa},\hatt c}$ if and only if $c=\hatt{c}$ and $0\leq\kappa\leq \hatt{\kappa}$. 

Let us also establish the necessary and sufficient conditions on the operator $B$ defined in equations \eqref{eq:3.17} and \eqref{eq:3.18} such that $T_{B,\mathcal{N}_0}\leq T_{\kappa,c}$. Since $\mathcal{W}_c\subsetneq \mathcal{N}_0$, this will be the case if
\begin{equation}
\langle \eta, B\eta\rangle_{L^2_r}\leq \langle \eta, \kappa\eta\rangle_{L^2_r}
\end{equation}
for every $\eta\in\mathcal{W}_c$. For $c=\infty$, this is equivalent to the condition
\begin{equation}
\langle \hatt v_a, B \hatt v_a\rangle_{L^2_r} = b_{22} \|\hatt v_a\|_{L^2_r}^2 \leq \kappa \|\hatt v_a\|_{L^2_r}^2 \quad \Leftrightarrow \quad b_{22}\leq \kappa.
\end{equation}
If $c\in\mathbb{C}$, we get the condition
\begin{align}
\langle \hatt v_a+cu_a, B(\hatt v_a+cu_a)\rangle_{L^2_r}&=(b_{22}+2\mbox{Re}(cb_{12}))\|\hatt v_a\|^2_{L^2_r}+|c|^2b_{11}\|u_a\|^2_{L^2_r}  \\
&\leq \langle \hatt v_a+cu_a,\kappa(\hatt v_a+cu_a)\rangle_{L^2_r}= \kappa \|\hatt v_a\|^2_{L^2_r}+|c|^2\kappa\|u_a\|^2_{L^2_r}.\notag
\end{align}

Now, by Theorem \ref{t3.1}, the operator $T_{\kappa,c}$ for $c\in\mathbb{C}$ has domain
\begin{align}
&\dom(T_{\kappa,c}) \notag\\
&\quad=\dom(T_{min})\dot{+}\{((T_F)^{-1}\kappa+I)\eta\:|\:\eta\in\mathcal{W}_c\}\dot{+}\{(T_F)^{-1}v\:|\:v\in\mathcal{N}_0\cap\mathcal{W}_c^\perp\}\\
&\quad=\dom(T_{min})\dot{+}\mbox{span}\left\{\kappa (\hatt \xi_a+c\xi_a)+\hatt{v}_a+cu_a\right\}\dot{+}\mbox{span}\left\{\overline{c}\frac{\hatt{\xi}_a}{\|\hatt v_a\|^2_{L_r^2}}-\frac{\xi_a}{\|u_a\|^2_{L_r^2}}\right\},\notag
\end{align}
while for $c=\infty$, the domain $\dom(T_{\kappa,\infty})$ is given by
\begin{equation}
\dom(T_{\kappa,\infty})=\dom(T_{min})\dot{+}\mbox{span}\left\{\kappa \xi_a+u_a\right\}\dot{+}\mbox{span}\{\hatt{\xi}_a\}.
\end{equation}

We then find the following generalized boundary conditions that the elements of $\dom(T_{max})$ must satisfy in order to also belong to $\dom(T_{\kappa,\infty})$:

\begin{align}
\wti{g}(a)=0,\quad 
\wti{g}^{\,\prime}(b)=\frac{\wti{g}(b)}{\wti u_a(0,b)}\left(\wti u_a^{\,\prime}(0,b)-\kappa\frac{\|u_a\|^2_{L^2_r}}{\wti u_a(0,b)}\right).
\end{align}
This describes separated boundary conditions with the Dirichlet-type boundary condition imposed at the left endpoint $a$ and a separated boundary condition at $b$ defined via \eqref{2.25} (see Theorem \ref{Thm:dim1} for explicit form).

The second case of separated boundary conditions in this setting occurs when imposing the Dirichlet-type boundary condition at the right endpoint. This corresponds to the special value $c_F=-\wti{\hatt v}_a(0,b)/\wti u_a(0,b)$, with the boundary conditions determining the elements of $\dom(T_{\kappa,c_F})$ given by the following:
\begin{align}
\wti g(b)=0,\quad
\wti g^{\,\prime}(a)&=\wti g(a)\bigg(\kappa \bigg[ \|\hatt v_a\|^2_{L^2_r}+\bigg(\frac{\wti{\hatt v}_a(0,b)}{\wti u_a(0,b)}\bigg)^2\|u_a\|^2_{L^2_r}\bigg]-\frac{\wti{\hatt v}_a(0,b)}{\wti u_a(0,b)}+\wti{\hatt v}_a^{\,\prime}(0,a)\bigg).
\end{align}

Notice that since the principal and nonprincipal solutions are real-valued, the only way that $\wti{\hatt v}_a(0,b)+d\wti u_a(0,b)=0,\ d\in\{c,\overline{c}\}$, would be if $c=c_F\in\bbR$.
Thus, the generalized boundary conditions for $c\in\mathbb{C}\setminus\{c_F\}$ can be written as coupled boundary conditions as in \eqref{2.26}.

We once again summarize these results.

\begin{theorem}\label{Thm:dim1}
Assume Hypothesis \ref{h2.1}, that $T_{min}$ is strictly positive, and that $\tau$ is in the limit circle case at $a$ and $b$. Let $\kappa\geq 0$ and $c\in\mathbb{C}\cup\{\infty\}$. If $\dim(\cW)=1$, then the separated boundary conditions for $g\in\dom(T_{max})$ describing nonnegative self-adjoint extensions of $T_{min}$ are given by
\begin{equation}\label{3.42}
\a=\pi,\quad \cot(\beta)=\frac{1}{\wti u_a(0,b)}\bigg(\wti u_a^{\,\prime}(0,b)-\kappa\frac{\|u_a\|^2_{L^2_r}}{\wti u_a(0,b)}\bigg),\ \beta\in(0,\pi),
\end{equation}
for $c=\infty$,
\begin{equation}\label{3.43}
\b=\pi,\quad \cot(\a)=\frac{\wti{\hatt v}_a(0,b)}{\wti u_a(0,b)}-\wti{\hatt v}_a^{\,\prime}(0,a)-\kappa \bigg[ \|\hatt v_a\|^2_{L^2_r}+\bigg(\frac{\wti{\hatt v}_a(0,b)}{\wti u_a(0,b)}\bigg)^2\|u_a\|^2_{L^2_r}\bigg],\ \a\in(0,\pi),
\end{equation}
for $c=c_F=-\wti{\hatt v}_a(0,b)/\wti u_a(0,b)$, and the coupled boundary conditions are given by
\begin{align}
\wti{g}(b)&=\wti{g}(a)\left(\wti{\hatt v}_a(0,b)+c\wti u_a(0,b)\right),\\
\wti{g}^{\,\prime}(b)&=\frac{\wti{g}^{\,\prime}(a)}{\wti{\hatt v}_a(0,b)+\overline{c}\wti u_a(0,b)}+\wti{g}(a)\Bigg[
\wti{\hatt v}_a^{\,\prime}(0,b)+c\wti{u}_a^{\,\prime}(0,b)-\kappa c\frac{\|u_a\|^2_{L_r^2}}{\wti u_a(0,b)} \\
&\qquad-\frac{1}{\wti{\hatt v}_a(0,b)+\overline{c}\wti u_a(0,b)}\bigg(\kappa\bigg(\|\hatt v_a\|^2_{L^2_r}-c\frac{\wti{\hatt v}_a(0,b)}{\wti u_a(0,b)}\|u_a\|^2_{L^2_r}\bigg)+c+\wti{\hatt v}_a^{\,\prime}(0,a)\bigg)\Bigg],   \notag
\end{align}
for $c\in\mathbb{C}\backslash\{c_F\}$. In particular, $R_{12}=0$ and $R_{22}=1/R_{11}$ in \eqref{2.26} for this case.

Moreover, $0\leq T_{\kappa,c}\leq T_{\hatt{\kappa},\hatt c}$ if and only if $c=\hatt{c}$ and $0\leq\kappa\leq \hatt{\kappa}$. In addition, $T_{B,\mathcal{N}_0}\leq T_{\kappa,c}$ if and only if $b_{22}\leq \kappa$ when $c=\infty$ or 
\begin{equation}
\langle \hatt v_a+cu_a, B(\hatt v_a+cu_a)\rangle_{L^2_r}\leq \kappa \|\hatt v_a\|^2_{L^2_r}+|c|^2\kappa\|u_a\|^2_{L^2_r},\quad c\in\mathbb{C}.
\end{equation}
\end{theorem}

Note that the separated boundary conditions can be easily written as admissible ranges and motivates $\pi$ for the Friedrichs extensions. In particular, one obtains
\begin{equation}\label{3.47}
\alpha=\pi,\ \beta\in\left[\cot^{-1}\left(\frac{\wti u^{\,\prime}_a(0,b)}{\wti u_a(0,b)}\right),\pi\right],\quad \text{or}\quad 
\alpha\in\left[\cot^{-1}\left(\frac{\wti{\hatt u}_a(0,b)}{\wti u_a(0,b)}\right),\pi\right],\ \beta=\pi,
\end{equation}
where we included the Friedrichs extension and used that (see \eqref{B.3.8}, \eqref{eq:3.9})
\begin{equation}
\frac{\wti{\hatt v}_a(0,b)}{\wti u_a(0,b)}-\wti{\hatt v}_a^{\,\prime}(0,a)=\frac{\wti{\hatt u}_a(0,b)}{\wti u_a(0,b)}.
\end{equation}

\subsubsection{The case \texorpdfstring{$\dim(\mathcal{W})=0$}{dim0}}
The last possibility is that $\dim(\mathcal{W})=0$, in which case $\mathcal{W}^\perp\cap\mathcal{N}_{0}=\mathcal{N}_0$. The unique self-adjoint extension corresponding to $\mathcal{W}=\{0\}$ is, of course, the Friedrichs extension, $T_F$, whose domain is given by (see \eqref{2.30})
\begin{equation}
\dom(T_F)=\dom(T_{min})\dot{+}\mbox{span}\{\xi_a,\hatt{\xi}_a\}=\{g\in\dom(T_{max})\:|\:\wti{g}(a)=\wti{g}(b)=0\}.
\end{equation}
By Theorem \ref{t3.1}, the choice $\mathcal{W}=\{0\}$ implies that $T_F$ is larger than any other nonnegative self-adjoint extension, that is, $\hatt{T}\leq T_F,$
for any other nonnegative self-adjoint extension $\hatt{T}$ of $T_{\min}$ as were described before.

\subsection{Nonnegative extensions for a fixed boundary condition}\label{subnonneg}

We now study nonnegative self-adjoint extensions of $T_{min}$ whenever not only $T_{min}\geq \varepsilon I$ for some $\varepsilon>0$, but one also has an extension with a fixed separated boundary condition at one endpoint, and Dirichlet-type at the other, is strictly positive. In the next section, we will apply the results proven here to Sturm--Liouville operators with symmetric coefficient functions.

Throughout this subsection we make the following assumption:

\begin{hypothesis}\label{h3.1}
Assume Hypothesis \ref{h2.1}, that $\tau$ is limit circle at both endpoints, and that for fixed $\beta'\in(0,\pi]$, $T_{\pi,\beta'}\geq \varepsilon I$ for some $\varepsilon>0$.
\end{hypothesis}

We now study nonnegative self-adjoint extensions of $T_{min,\beta'}$, that is, we fix the separated $\beta'$ boundary condition at $x=b$ (see \eqref{2.25}) and study extensions of the associated minimal operator,
\begin{align}
& T_{min,\be '} f = \tau f, \notag\\
& f \in \dom(T_{min,\be '})=\big\{g\in\dom(T_{max}) \, \big| \, \wti g(a)= \wti g^{\, \prime}(a)=0;  \\ 
& \hspace*{6.0cm} \, \wti g(b)\cos(\be')- \wti g^{\, \prime}(b)\sin(\be') = 0 \big\}.    \notag
\end{align}
Note that the Friedrichs extension of $T_{min,\beta'}$ is exactly $T_{\pi,\beta'}$, and $T_{min,\beta'}\geq\varepsilon I$ by assumption. For details regarding the case when $x=b$ is instead in the limit point case, see Remark \ref{Rem:3.3} $(iv)$.

Notice that by fixing the boundary condition at $x=b$, the minimal operator $T_{\min,\be '}$ now has deficiency index 1. 
From the fact that the deficiency index is equal to 1 and that $T_{min,\be'}\geq \varepsilon I$ for some $\varepsilon>0$, it immediately follows that $\dim(\mathcal{N}_0)=1$. Hence, in the characterization of all nonnegative self-adjoint extensions of $T_{min,\be'}$ in Theorem \ref{t3.1}, there are only two possibilities: either $\mathcal{W}=\{0\}$ or $\mathcal{W}=\mathcal{N}_0$. In the first case, $\mathcal{W}=\{0\}$, we obtain the Friedrichs extension $(T_{min,\beta'})_F$, which is given by $T_{\pi,\beta'}$. Hence, we consider the case $\mathcal{W}=\mathcal{N}_0$ from now on. In this case, by Theorem \ref{t3.1}, all nonnegative self-adjoint extensions of $T_{min,\be'}$ with $\mathcal{W}=\mathcal{N}_0$ are parametrized by nonnegative self-adjoint operators $B:\mathcal{W}=\mathcal{N}_0\rightarrow\mathcal{N}_0$. Since $B$ acts in a one-dimensional space, it is a multiplication by a nonnegative scalar $B\geq 0$. Thus, all other nonnegative self-adjoint extensions $(T_{min,\be'})_{B,\mathcal{N}_0}$ of $T_{min,\be'}$ that are not equal to its Friedrichs extension, have domain given by
\begin{equation}
\dom((T_{min,\be'})_{B,\mathcal{N}_0})=\dom(T_{min,\be'})\dot{+}\mbox{span}\{(T_{min,\beta'})_F^{-1}B\eta_{\be'}+\eta_{\be'}\},
\end{equation}
where, \emph{a priori}, $\eta_{\be'}$ is any vector that spans $\mathcal{N}_0=\ker(T_{min,\be'}^*)$. Next, we make a  specific choice for $\eta_{\be'}$.

For the fixed $\beta'$ in Hypothesis \ref{h3.1}, choose any principal and nonprincipal solutions of $\tau u=0$ at $x=a$ that are solutions on $(a,b)$ with Wronskian 1 and define
\begin{equation}
\eta_{\beta'}(x)=\hatt u_a(0,x)-\frac{\cos(\beta')\wti{\hatt u}_a(0,b)-\sin(\beta')\wti{\hatt u}_a^{\,\prime}(0,b)}{\cos(\beta')\wti{u}_a(0,b)-\sin(\beta')\wti{u}_a^{\,\prime}(0,b)} u_a(0,x),\quad x\in(a,b).
\end{equation}
Note that Hypothesis \ref{h3.1} guarantees that the denominator is not zero as otherwise $0$ would be an eigenvalue of $T_{\pi,\beta'}$, so that by construction, $\eta_{\beta'}$ satisfies the $\beta'$ boundary condition at $x=b$ and $\tau\eta_{\beta'}=0$, that is, $\eta_{\beta'}\in\ker(T_{max,\beta'})$. Furthermore, defining the generalized boundary values at $x=a$ in \eqref{2.20} and \eqref{2.21} using these choices of $u_a,$ $\hatt u_a$, one readily verifies that $\wti\eta_{\beta'}(a)=1$ and
\begin{equation}
\wti\eta_{\beta'}^{\,\prime}(a)=-\frac{\cos(\beta')\wti{\hatt u}_a(0,b)-\sin(\beta')\wti{\hatt u}_a^{\,\prime}(0,b)}{\cos(\beta')\wti{u}_a(0,b)-\sin(\beta')\wti{u}_a^{\,\prime}(0,b)}.
\end{equation}

Now to find an expression for $\xi_{\beta'}(x):=(T_F)^{-1}(\eta_{\beta'})(x)$ in terms of principal and nonprincipal solutions, we note that $\xi_{\beta'}$ must satisfy the $\beta'$ boundary condition at $x=b$, $\wti\xi_{\beta'}(a)=0$, and
\begin{equation}
\tau \xi_{\beta'}(x)=\eta_{\beta'}(x),\quad x\in(a,b).
\end{equation}
Using the variation of constants formula allows one to construct the representation
\begin{align}
\xi_{\beta'}(x)&=\hatt u_a(0,x)\int_a^x u_a(0,t)\eta_{\beta'}(t)r(t)dt \notag\\
&\quad + u_a(0,x)\bigg(\int_x^b \hatt u_a(0,t)\eta_{\beta'}(t)r(t)dt\\
&\hspace{2.5cm} -\frac{\cos(\beta')\wti{\hatt u}_a(0,b)-\sin(\beta')\wti{\hatt u}_a^{\,\prime}(0,b)}{\cos(\beta')\wti{u}_a(0,b)-\sin(\beta')\wti{u}_a^{\,\prime}(0,b)}\int_a^b u_a(0,t)\eta_{\beta'}(t)r(t)dt\bigg). \no
\end{align}
One readily verifies that all of the required properties above (boundary conditions and differential equation) are satisfied. In addition, direct calculation yields
\begin{equation}
\wti\xi_{\beta'}^{\,\prime}(a)=\int_a^b \eta_{\beta'}^2(t)r(t)dt=\norm{\eta_{\beta'}}^2_{L^2_r}>0.
\end{equation}
Therefore, letting $h_B(x)=B\xi_{\be'}(x)+\eta_{\be'}(x)$, we can write
\begin{equation}
\dom((T_{min,\beta'})_{B,\mathcal{N}_0})=\dom(T_{min,\beta})\dot{+}\mbox{span}\{h_B\}.
\end{equation}

To express $\dom((T_{min,\beta'})_{B,\mathcal{N}_0})$ in terms of boundary conditions, that is, to find $\alpha$ such that $(T_{min,\beta'})_{B,\mathcal{N}_0}=T_{\alpha,\be'}$, we need to understand the behavior of $h_B$ at the left endpoint $a$:
\begin{align}
\wti h_B(a)&=1,\\
\wti h_B^{\,\prime}(a)&=B\norm{\eta_{\beta'}}^2_{L^2_r}-\frac{\cos(\beta')\wti{\hatt u}_a(0,b)-\sin(\beta')\wti{\hatt u}_a^{\,\prime}(0,b)}{\cos(\beta')\wti{u}_a(0,b)-\sin(\beta')\wti{u}_a^{\,\prime}(0,b)},\quad B\geq0. \notag
\end{align}
Hence by considering the equation defining extensions satisfying $\alpha$ boundary conditions at $x=a$ in \eqref{2.25}, we conclude that any $\alpha\in(0,\pi)$ that satisfies
\begin{equation}\label{3.13}
\cos(\alpha)+\left(B\norm{\eta_{\beta'}}^2_{L^2_r}-\frac{\cos(\beta')\wti{\hatt u}_a(0,b)-\sin(\beta')\wti{\hatt u}_a^{\,\prime}(0,b)}{\cos(\beta')\wti{u}_a(0,b)-\sin(\beta')\wti{u}_a^{\,\prime}(0,b)}\right) \sin(\alpha)=0,
\end{equation}
for $B\geq0$, defines a nonnegative extension that will add the appropriate functions, namely, will include $h_B$ in their domain. Rearranging \eqref{3.13} yields for $\alpha\in(0,\pi)$,
\begin{equation}\label{3.14}
\alpha=\alpha(B)=\cot^{-1}\left(\frac{\cos(\beta')\wti{\hatt u}_a(0,b)-\sin(\beta')\wti{\hatt u}_a^{\,\prime}(0,b)}{\cos(\beta')\wti{u}_a(0,b)-\sin(\beta')\wti{u}_a^{\,\prime}(0,b)}-B\norm{\eta_{\beta'}}^2_{L^2_r}\right),\ B\geq0,
\end{equation}
that is, $(T_{min,\beta'})_{B,\mathcal{N}_0}=T_{\alpha(B),\be'}$.

This allows one to see that the case $B=\infty$ corresponds to the Friedrichs extension (i.e., $\alpha=\pi$) of $T_{min,\beta'}$ and the case $B=0$ corresponds to the Krein--von Neumann extension of $T_{min,\beta'}$ at the beginning of the admissible $\alpha$-range. 

Moreover, the parameter $B\geq0$ is known to order the extensions (see Theorem \ref{t3.1}), that is,  $0\leq B_1\leq B_2$ implies that $(T_{min,\beta'})_{B_1,\mathcal{N}_0}\leq (T_{min,\beta'})_{B_2,\mathcal{N}_0}\leq (T_{min,\beta'})_{F} $. From \eqref{3.14} it follows that $\alpha$ is monotone increasing in $B$, so $B_1\leq B_2$ implies $\alpha(B_1)\leq \alpha(B_2)$, so the extensions defined by these choices of $\alpha$ inherit that same ordering when moving from left to right in the admissible $\alpha$-range.

We summarize these results in the following theorem.

\begin{theorem}
Assume Hypothesis \ref{h3.1}. Then the generalized boundary conditions at $x=a$ for $g\in\dom(T_{max})$ describing nonnegative self-adjoint extensions of $T_{min,\beta'}$ are given by \eqref{2.25} with
\begin{equation}\label{3.15}
\alpha\in\left[\cot^{-1}\left(\frac{\cos(\beta')\wti{\hatt u}_a(0,b)-\sin(\beta')\wti{\hatt u}_a^{\,\prime}(0,b)}{\cos(\beta')\wti{u}_a(0,b)-\sin(\beta')\wti{u}_a^{\,\prime}(0,b)}\right),\pi\right].
\end{equation}
Moreover,
\begin{equation} \label{eq:3.19}
\cot^{-1}\left(\frac{\cos(\beta')\wti{\hatt u}_a(0,b)-\sin(\beta')\wti{\hatt u}_a^{\,\prime}(0,b)}{\cos(\beta')\wti{u}_a(0,b)-\sin(\beta')\wti{u}_a^{\,\prime}(0,b)}\right)\leq \alpha_1\leq \alpha_2 \quad \Leftrightarrow \quad 0\leq T_{\alpha_1,\be'}\leq T_{\alpha_2,\be'}.
\end{equation}
\end{theorem}

We point out that all of these extensions are in fact strictly positive except for when $\alpha$ is chosen to be equal to the lower bound of the interval in \eqref{3.15}. This follows from general results such as \cite[Cor. 2.14]{AS80}, but can also be easily seen from setting the constant term equal to zero in the small-$z$ expansion of $F_{\a,\b'}(z)$ in \eqref{2.43}.

A few remarks are now in order. 
\begin{remark}\label{Rem:3.3}
$(i)$ Equation \eqref{3.14} (see also \eqref{3.42}, \eqref{3.43}) leads to the natural identification of the Friedrichs extension with  the choice $\alpha=\pi$ rather than the typical $\alpha=0$ in order to preserve the ordering of the extensions described in Theorem \ref{t3.1}.\\[1mm]
$(ii)$ Since the principal and nonprincipal solutions are not unique, one should study what happens to \eqref{3.15} whenever choosing different solutions. Clearly, varying the $u_b,\hatt u_b$, chosen in defining the boundary values at $x=b$ changes the boundary values given in \eqref{3.15}. Moreover, choosing a different constant multiple, $c$, of the principal solution $u_a$ will change the nonprincipal solution $\hatt u_a$ by $1/c$. The nonprincipal solution can also be changed by an additive constant multiple, $d$, of the principal solution. As a result,
\begin{align}
\begin{split}
& u_a(0,\dott) \longrightarrow c u_a(0,\dott) \text{ and } \hatt u_a (0, \dott) \longrightarrow \hatt u_a (0, \dott) + d u_a(0,\dott), \quad c\in\bbR\backslash\{0\},\ d  \in \bbR,\\
&\text{then } \cot^{-1}\left(\frac{\cos(\beta')\wti{\hatt u}_a(0,b)-\sin(\beta')\wti{\hatt u}_a^{\,\prime}(0,b)}{\cos(\beta')\wti{u}_a(0,b)-\sin(\beta')\wti{u}_a^{\,\prime}(0,b)}\right)\\
&\qquad\ \longrightarrow
\cot^{-1}\left(\frac{1}{c^2}\frac{\cos(\beta')\wti{\hatt u}_a(0,b)-\sin(\beta')\wti{\hatt u}_a^{\,\prime}(0,b)}{\cos(\beta')\wti{u}_a(0,b)-\sin(\beta')\wti{u}_a^{\,\prime}(0,b)}+d\right), \quad c\in\bbR\backslash\{0\},\ d  \in \bbR.
\end{split}
\end{align}
This shows that adjusting $u_a,\hatt u_a$, effectively adjusts the admissible choices of $\alpha$ above. In particular, one can always choose
\begin{equation}
d=-\frac{1}{c^2}\frac{\cos(\beta')\wti{\hatt u}_a(0,b)-\sin(\beta')\wti{\hatt u}_a^{\,\prime}(0,b)}{\cos(\beta')\wti{u}_a(0,b)-\sin(\beta')\wti{u}_a^{\,\prime}(0,b)},
\end{equation}
to make the $\alpha$-range simply become $[\pi/2,\pi]$ at the expense of a likely much more complicated nonprincipal solution. Furthermore, this will probably differ from the typical choice that guarantees the boundary values coincide with the regular boundary values in parameterized examples that include regular problems as special cases, such as the classic Bessel (see Section \ref{s5}) or Jacobi differential equations.\\[1mm]
$(iii)$ One could start instead by defining $\eta_{\beta'}(x)$ in terms of the principal and nonprincipal solutions at $x=b$ by letting
\begin{equation}\label{3.64}
\eta_{\beta'}(x)=\sin(\beta')\hatt u_b(0,x)+\cos(\beta')u_b(0,x),\quad x\in(a,b).
\end{equation}
Then one can proceed as before by finding a representation for $\xi_{\beta'}(x)$ using the variation of constants formula, though the form will be more complicated now. In particular, a careful analysis shows that the analog of \eqref{3.14} becomes, for $\alpha\in(0,\pi)$,
\begin{align}\label{3.19}
\begin{split}
\alpha=\cot^{-1}\Bigg(&-\frac{\cos(\beta')\wti{ u}^{\,\prime}_b(0,a)+\sin(\beta')\wti{\hatt u}_b^{\,\prime}(0,a)}{\cos(\beta')\wti{u}_b(0,a)+\sin(\beta')\wti{\hatt u}_b(0,a)}\\
&\quad-\frac{B\norm{\eta_{\beta'}}^2_{L^2_r}}{(\cos(\beta')\wti{u}_b(0,a)+\sin(\beta')\wti{\hatt u}_b(0,a))^2}\Bigg),\quad B\geq0.
\end{split}
\end{align}
Notice that the term multiplying $B$ in \eqref{3.19} is once again strictly negative, thus the ordering follows as before, and the term in the denominator is nonzero as otherwise $\wti\eta_{\beta'}(a)=0$ contradicting that 0 is not an eigenvalue of the Friedrichs extension. 

Letting $B=0$ in \eqref{3.19} gives an alternate representation of the Krein--von Neumann extension than before. In fact, setting $B=0$ and $\beta'=\pi$ in \eqref{3.19} gives the analog of \eqref{2.31} when replacing the limit point assumption at $x=b$ with fixing the Friedrichs extension boundary condition at $x=b$.

One can also prove this alternate representation by implementing the following identities which follow immediately from the Wronskian definition of the generalized boundary values in \eqref{2.20} and \eqref{2.21}:
\begin{equation}
\wti u_a(0,b)=-\wti u_b(0,a),\ \wti {\hatt u}_a^{\,\prime}(0,b)=-\wti{\hatt u}_b^{\,\prime}(0,a),\ \wti{\hatt u}_a(0,b)=\wti u_b^{\,\prime}(0,a),\ \wti{\hatt u}_b(0,a)=\wti u_a^{\,\prime}(0,b).
\end{equation}
Substituting these values into \eqref{3.14} now immediately yields \eqref{3.19}.\\[1mm]
$(iv)$ If one assumes instead that $x=b$ is in the limit point case, the analysis using the principal and nonprincipal solutions at $x=b$ in $(iii)$ can be applied with $\beta'=\pi$ since this is simply choosing the square integrable solution at $x=b$ for $\eta(x)$ in \eqref{3.64}. Thus whenever $b$ is in the limit point case our analysis, specifically \eqref{3.19}, recovers \eqref{2.31} as well as \cite[Thm. 2.1]{BE16}, noting that $\psi$ in \cite[Thm. 2.1]{BE16} is equal to $u_b(0,x)/\wti u_b(0,a)$ in our notation. As an explicit comparison, we point out that \cite{BE16} implicitly assumes $\wti\eta_\pi(a)=\wti u_b(0,a)=1$ whereas our analysis does not. Therefore \eqref{3.19} includes factors of $\wti u_b(0,a)$ in the denominators that do not appear when comparing to \cite[Eq. (2.1)]{BE16} (utilizing the left endpoint equation in \eqref{2.25}). \\[1mm]
$(v)$ Finally, we remark that similar results hold when considering the boundary condition at the left endpoint fixed instead. In particular, fixing $\alpha'\in(0,\pi]$, the only additional difference other than interchanging $a$ with $b$ and $\beta'$ with $\alpha'$ throughout is that $\sin(\beta')$ is replaced with $-\sin(\alpha')$ due to the signs chosen in \eqref{2.25}. This also causes a sign change in the argument of the cotangent above.
\end{remark}

We will apply the results from this section in the following one when investigating Sturm--Liouville operators with symmetric coefficient functions.

\section{An application to Sturm--Liouville operators with coefficient functions symmetric about the midpoint of the interval}\label{s4}

We now make the following additional assumptions regarding the symmetries of the coefficient functions of the Sturm--Liouville operators being considered.

\begin{hypothesis}\label{h2.13}
Assume Hypothesis \ref{h2.1}. In addition, assume that $a$ and $b$ are finite, $\tau$ is limit circle nonoscillatory at both endpoints, and suppose that the coefficient functions $p,q,r$ are symmetric about $x=(b+a)/2$ such that
\begin{equation}
p(x)=p(b-x+a),\ q(x)=q(b-x+a),\ r(x)=r(b-x+a),\quad x\in(a,b).
\end{equation}
\end{hypothesis}

Next, we return to the self-adjoint extensions defined using generalized boundary conditions in Theorem \ref{t2.10} and explore which of these conditions commute with the unitary operator, $P$, that reflects the Sturm--Liouville problem about the midpoint of the interval considered. Notice that this will be equivalent to interchanging the roles of the endpoints and negating the derivative terms in \eqref{2.25} for separated and \eqref{2.26} for coupled boundary conditions. This results in the following.

\begin{proposition}\label{t.2.14}
Assume Hypothesis \ref{h2.13}. Then the family self-adjoint extensions with separated boundary conditions that are invariant with respect to the symmetry transformation $P$ $($i.e., reflection about the midpoint of the interval$)$ are given by setting $\alpha=\beta$ in \eqref{2.25}$;$ the family with coupled boundary conditions that are invariant with respect to $P$ are given by $\eta=0$ and $R_{11}=R_{22}$ in \eqref{2.26}.
\end{proposition}
\begin{proof}
The separated boundary condition statement follows immediately from interchanging the endpoints and negating the derivative terms in \eqref{2.25}, then comparing the new domain to the original domain.

For coupled boundary conditions, following the same procedure leads to comparing the following to the original domain description in \eqref{2.26}:
\begin{equation}
\begin{pmatrix} 1 & 0\\ 0 & -1\end{pmatrix} \begin{pmatrix} \wti g(b)\\ \wti g^{\, \prime}(b)\end{pmatrix} =\begin{pmatrix} \wti g(b)\\ -\wti g^{\, \prime}(b)\end{pmatrix} 
= e^{-i\eta}R^{-1} \begin{pmatrix} 1 & 0\\ 0 & -1\end{pmatrix} \begin{pmatrix}
\wti g(a)\\ \wti g^{\, \prime}(a)\end{pmatrix}. 
\end{equation}
Hence, the problem reduces to considering when
\begin{equation}
e^{i\eta}R=\begin{pmatrix} 1 & 0\\ 0 & -1\end{pmatrix}e^{-i\eta}R^{-1} \begin{pmatrix} 1 & 0\\ 0 & -1\end{pmatrix},\quad \eta\in[0,\pi),\ R\in SL(2,\bbR),
\end{equation}
which yields the requirements $\eta=0$ and $R_{11}=R_{22}$.
\end{proof}
We would like to point out that our choice of the negative sign on the generalized derivative at $b$ in \eqref{2.25} guarantees $\alpha=\beta$ in the previous theorem. If a different parameterization were chosen, a negative would need to be introduced to account for the symmetries in the derivative.

Notice that many classical boundary conditions lead to self-adjoint extensions of the form given in Proposition \ref{t.2.14}: the Friedrichs extension ($T_F=T_{\pi,\pi}$), the periodic extension ($T_{Per}=T_{0,I_2}$), and the antiperiodic extension ($T_{Anti}=T_{0,-I_2}$). On the other hand, some extensions will only be of this form under additional assumptions. For example, when considering the Krein--von Neumann extension described in Theorem \ref{t2.12} $(ii)$, one would need to additionally assume that $\wti{\hatt u}_{a}(0,b) = \wti u_{a}^{\, \prime}(0,b)$, which is not generically true. But if the coefficient functions are themselves symmetric, this is always true by \eqref{4.29}, showing the Krein--von Neumann extension is of this form under Hypothesis \ref{h2.13} as expected from the general theory.

One can now prove the following decomposition when considering separated boundary conditions that are invariant with respect to reflection about the midpoint of the interval. Note that this decomposition allows one to readily use the results from Section \ref{subnonneg} to explicitly characterize all nonnegative self-adjoint extensions.

\begin{theorem}\label{t2.16}
Assume Hypothesis \ref{h2.13}. Let $T_{\alpha,\alpha}$ denote the $(\alpha,\alpha)$-extension of the minimal operator associated with $\tau$, let $T^{(a,(b+a)/2)}_{\alpha,\pi}$ denote the $\alpha$--Dirichlet-type extension of the minimal operator associated with $\tau\big|_{(a,(b+a)/2)}$, and let $T^{(a,(b+a)/2)}_{\alpha,\pi/2}$ denote the $\alpha$--Neumann-type extension of the minimal operator associated with $\tau\big|_{(a,(b+a)/2)}$. Then 
$T_{\alpha,\alpha}$ is unitarily equivalent to the direct sum 
\begin{equation} \label{eq:4.12.0}
T^{(a,(b+a)/2)}_{\alpha,\pi}\oplus T^{(a,(b+a)/2)}_{\alpha,\pi/2}
\end{equation}
in $L^2((a,(b+a)/2),rdx)\oplus L^2((a,(b+a)/2),rdx)$, where $\alpha\in(0,\pi]$.
In particular,
\begin{equation}\label{4.12}
\sigma(T_{\alpha,\alpha})=\sigma(T^{(a,(b+a)/2)}_{\alpha,\pi})\cup\sigma(T^{(a,(b+a)/2)}_{\alpha,\pi/2}),\quad \alpha\in(0,\pi].
\end{equation}

Moreover, if $T_{min}\geq\varepsilon I$ for some $\varepsilon>0$, then $T_{\a,\a}$ is nonnegative for $\alpha\in[\nu,\pi]$ $($in fact, positive for $\alpha\in(\nu,\pi]$$)$ where
\begin{equation}\label{4.14}
\nu=\cot^{-1}\left(\frac{\hatt u^{[1]}_a(0,(b+a)/2)}{u_a^{[1]}(0,(b+a)/2)}\right).
\end{equation}
Finally, 
\begin{equation} \label{eq:4.15}
\quad \nu\leq\alpha_1\leq \alpha_2 \quad\mbox{if and only if}\quad 0\leq T_{\alpha_1,\alpha_1}\leq T_{\alpha_2,\alpha_2}.
\end{equation}
\end{theorem}
\begin{proof}
Firstly, note that $P$ has eigenvalues $\pm1$ with the corresponding eigenspaces being the elements of $L^2((a,b),rdx)$ that are symmetric/antisymmetric with respect to reflection about the midpoint $(a+b)/2$ of the interval $(a,b)$. In other words, we have
\begin{equation}
\ker(P\mp I)=\{f\pm Pf | f\in L^2((a,b);rdx)\}.
\end{equation}
Next, by the assumptions made on the coefficient functions $p,q,r$, the operators $T_{\alpha,\alpha}$ and $P$ commute, that is, if $f\in\dom(T_{\alpha,\alpha})$, then so is $Pf$ and we have
\begin{equation}
T_{\alpha,\alpha}Pf=PT_{\alpha,\alpha}f
\end{equation}
for every $f\in\dom(T_{\alpha,\alpha})$. The eigenspaces $\ker(P\mp I)$ therefore reduce the operator $T_{\alpha,\alpha}$, which can be decomposed into 
\begin{equation}\label{4.10a}
T_{\alpha,\alpha}=T_{\alpha,\alpha}^+\oplus T_{\alpha,\alpha}^-,
\end{equation}
where $T_{\alpha,\alpha}^\pm:=T_{\alpha,\alpha}\big|_{\ker(P\mp I)}$.
We now claim that $T_{\alpha,\alpha}^+=U_+^*T^{(a,(b+a)/2)}_{\alpha,\pi/2}U_+$ and $T_{\alpha,\alpha}^-=U_-^*T^{(a,(b+a)/2)}_{\alpha,\pi}U_-$, where the unitary operators $U_\pm$ are given by
\begin{align}
\begin{split}
U_\pm:\quad \ker(P\mp I)&\rightarrow L^2((a,(a+b)/2);rdx),\\    
f&\mapsto \sqrt{2}f|_{(a,(a+b)/2)},
\end{split}
\end{align}
with inverse $U_{\pm}^{-1}=U_\pm^*$ given by
\begin{align}
\begin{split}
U_\pm^*:\quad &L^2((a,(a+b)/2);rdx)\rightarrow \ker(P\mp I),\\
&(U_\pm^*g)(x)=\frac{1}{\sqrt{2}}\begin{cases} g(x), &\mbox{if}\quad x\in(a,(a+b)/2), \\
\pm g(b-x+a),& \mbox{if}\quad x\in((a+b)/2,b).\end{cases}
\end{split}
\end{align}
We begin by showing $U_+\dom(T_{\alpha,\alpha}^+)\subseteq\dom(T_{\alpha,\pi/2}^{(a,(b+a)/2)})$. If $f\in\dom(T^+_{\alpha,\alpha})=\dom(T_{\alpha,\alpha})\cap\ker(P-I)$, note that $f(x)=f(b-x+a)$ and $(pf')(x)=-(pf')(b-x+a)$ for almost every $x\in(a,b)$ and $\widetilde{f}(a)\cos(\alpha)+\widetilde{f}'(a)\sin(\alpha)=0$. Moreover, since $f, pf'\in AC_{loc}(a,b)$, this implies that the following limits exist and are equal
\begin{align}
\begin{split}
\lim_{x\nearrow (a+b)/2}(pf')(x)=&\lim_{x\searrow (a+b)/2}(pf')(x)\\=&-\lim_{x\searrow (a+b)/2}(pf')(b-x+a)=-\lim_{x\nearrow (a+b)/2}(pf')(x),
\end{split}
\end{align}
which is only possible if $(pf')((a+b)/2)=0$. Hence, $U_+f$, which is proportional to the restriction of $f$ to the interval $(a,(a+b)/2)$ satisfies the boundary condition
\begin{equation}
\left(\widetilde{U_+f}\right)(a)\cos(\alpha)+\left(\widetilde{U_+f}\right)'(a)\sin(\alpha)=0
\end{equation}
and the Neumann boundary condition $(p(U_+f)')((a+b)/2)=0$. The regularity properties of $f$ on $(a,b)$ are also inherited by $U_+f$. We therefore conclude $U_+\dom(T_{\alpha,\alpha}^+)\subseteq\dom(T_{\alpha,\pi/2}^{(a,(b+a)/2)})$. The other inclusion $U_+^*\dom(T_{\alpha,\pi/2}^{(a,(b+a)/2)})\subseteq\dom(T_{\alpha,\alpha}^+)$ follows from similar considerations: if $f\in \dom(T_{\alpha,\pi/2}^{(a,(a+b)/2)})$ satisfies the boundary condition parameterized by $\alpha$ at $a$ and a Neumann condition at $(a+b)/2$, then $U_+^*f$ and $p(U^*_+f)'$ are both in $AC_{loc}(a,b)$ and will satisfy the boundary conditions at both endpoints of $(a,b)$. Lastly, it immediately follows that $U_+\tau f=\tau^{(a,(a+b)/2} U_+f$ and consequently $T_{\alpha,\alpha}^+=U^*_+T_{\alpha,\pi/2}^{(a,(b+a)/2)}U_+$.

The argument proving that $T_{\alpha,\alpha}^-=U_-^*T_{\alpha,\pi}^{(a,(b+a)/2)}U_-$ follows completely analogously noting that since $f\in\dom(T_{\alpha,\alpha})\cap\ker(P+I)$ implies $f(x)=-f(b-x+a)$ for all $x\in (a,b)$ and consequently 
\begin{align}
\begin{split}
\lim_{x\nearrow (a+b)/2}f(x)=&\lim_{x\searrow (a+b)/2}f(x)\\=&-\lim_{x\searrow (a+b)/2}f(b-x+a)=-\lim_{x\nearrow (a+b)/2}f(x),
\end{split}
\end{align}
from which the Dirichlet condition $f((a+b)/2)=0$ follows.
This finishes the first part of the proof from which the identity \eqref{4.12} immediately follows.

Next, note that a direct sum of two operators $T^{(a,(b+a)/2)}_{\alpha,\pi}\oplus T^{(a,(b+a)/2)}_{\alpha,\pi/2}$ is nonnegative if and only if each of its terms is.
Also, by Remark \ref{r3.5} $(iv)$, if $T^{(a,(b+a)/2)}_{\alpha,\pi/2}$ is a nonnegative self-adjoint extension, then so is $T^{(a,(b+a)/2)}_{\alpha,\pi}$.
Therefore $T_{\alpha,\alpha}$ is nonnegative if and only if $T^{(a,(b+a)/2)}_{\alpha,\pi/2}$ is. The necessary and sufficient condition for this is $\alpha\in[\nu,\pi]$, where $\nu$ is given by \eqref{4.14}, which is a particular case of \eqref{3.15} with $\beta'=\pi/2$. 

Lastly, from $\eqref{eq:3.19}$ it follows that $\nu\leq \alpha_1\leq\alpha_2$ implies $0\leq T^{(a,(b+a)/2)}_{\alpha_1,\pi/2}\leq T^{(a,(b+a)/2)}_{\alpha_2,\pi/2}$ and $T^{(a,(b+a)/2)}_{\alpha_1,\pi}\leq T^{(a,(b+a)/2)}_{\alpha_2,\pi}$. Since $T_{\alpha_{1/2},\alpha_{1/2}}$ is unitarily equivalent to the direct sum $T^{(a,(b+a)/2)}_{\alpha_{1/2},\pi/2}\oplus T^{(a,(b+a)/2)}_{\alpha_{1/2},\pi}$, this shows \eqref{eq:4.15}.
\end{proof}

As special cases, we point out the Friedrichs and Neumann-type extensions:
\begin{align}\label{4.15}
\begin{split}
&T_F=T_{\pi,\pi}\ \text{ is unitarily equivalent to }\ T^{(a,(b+a)/2)}_F\oplus T^{(a,(b+a)/2)}_{\pi,\pi/2},\\
&T_{\pi/2,\pi/2}\ \text{ is unitarily equivalent to }\ T^{(a,(b+a)/2)}_{\pi/2,\pi/2}\oplus T^{(a,(b+a)/2)}_{\pi/2,\pi}.
\end{split}
\end{align}

Next, we consider coupled boundary conditions that are invariant with respect to reflection about the midpoint of the interval. Once again, we use the results from Section \ref{subnonneg} to explicitly characterize all nonnegative self-adjoint extensions.

\begin{theorem}\label{t2.17}
Assume Hypothesis \ref{h2.13}. Let $T_{0,R'}$ denote the extension of the minimal operator associated with $\tau$ satisfying coupled boundary conditions $\eta=0$ and $R_{11}=R_{22}$ in \eqref{2.26}.
In addition, let $T^{(a,(b+a)/2)}_{\alpha,\pi}$ denote the $\alpha$--Dirichlet-type extension of the minimal operator associated with $\tau\big|_{(a,(b+a)/2)}$, and let $T^{(a,(b+a)/2)}_{\alpha',\pi/2}$ denote the $\alpha'$--Neumann-type extension of the minimal operator associated with $\tau\big|_{(a,(b+a)/2)}$. Then 
$T_{0,R'}$ is unitarily equivalent to the direct sum 
\begin{equation} \label{eq:4.16.0}
T^{(a,(b+a)/2)}_{\alpha,\pi}\oplus T^{(a,(b+a)/2)}_{\alpha',\pi/2}
\end{equation}
in $L^2((a,(b+a)/2),rdx)\oplus L^2((a,(b+a)/2),rdx)$,
where $\alpha,\alpha'\in(0,\pi]$ are given by
\begin{align}\label{4.17}
\begin{cases} \alpha=\cot^{-1}\left(\frac{R_{11}+1}{R_{12}}\right), \: \alpha'=\cot^{-1}\left(\frac{R_{11}-1}{R_{12}}\right),& \text{for }R_{12}\neq0,\\[3mm]
\alpha=\cot^{-1}\left(\frac{-R_{21}}{2}\right), \: \alpha'=\pi,& \text{for }R_{12}=0, R_{11}=-1,\\[3mm]\alpha=\pi, \: \alpha'=\cot^{-1}\left(\frac{R_{21}}{2}\right),& \text{for }R_{12}=0, R_{11}=1.
\end{cases}
\end{align}
In particular,
\begin{equation}\label{4.16}
\sigma(T_{0,R'})=\sigma(T^{(a,(b+a)/2)}_{\alpha,\pi})\cup\sigma(T^{(a,(b+a)/2)}_{\alpha',\pi/2}),
\end{equation}

Moreover, if $T_{min}\geq\varepsilon I$ for some $\varepsilon>0$, then $T_{0,R'}$ is nonnegative if both $\alpha\in[\nu,\pi]$ and $\alpha'\in[\mu,\pi]$ where $\alpha,\alpha'$ are defined by \eqref{4.17} and
\begin{equation}\label{4.24a}
\nu=\cot^{-1}\left(\frac{\hatt u_a(0,(b+a)/2)}{u_a(0,(b+a)/2)}\right),\quad
\mu=\cot^{-1}\left(\frac{\hatt u_a^{[1]}(0,(b+a)/2)}{u_a^{[1]}(0,(b+a)/2)}\right).
\end{equation}
The above extensions are strictly positive if and only if both $\alpha\in(\nu,\pi]$ and $\alpha'\in(\mu,\pi]$.
Finally, for
\begin{equation}
R'=\begin{pmatrix} R_{11} & R_{12}\\ R_{21} & R_{11} \end{pmatrix}\mbox{ and } \hat{R}'=\begin{pmatrix} \hat{R}_{11} & \hat{R}_{12}\\ \hat{R}_{21} & \hat{R}_{11} \end{pmatrix} \text{ with } \det(R')=\det(\hat{R}')=1,
\end{equation}
with both extensions $T_{R',0}, T_{\hat{R}',0}$ nonnegative, one has $T_{R',0}\leq T_{\hat{R}',0}$ if and only if one of the following cases hold$:$
\begin{itemize}
\item[$(i)$] $R_{12}\neq 0, \hat{R}_{12}\neq 0$, and
\begin{equation}\label{4.29a}
\frac{\hat{R}_{11}-1}{\hat{R}_{12}}\leq \frac{{R}_{11}-1}{{R}_{12}}\quad\mbox{and}\quad\frac{\hat{R}_{11}+1}{\hat{R}_{12}}\leq \frac{R_{11}+1}{R_{12}} ;
\end{equation}
\item[$(ii)$] $R_{12}\neq 0$, $\hat{R}_{12}=0$, $\hat{R}_{11}=1$, and 
\begin{equation}
 \frac{\hat{R}_{12}}{2}\leq \frac{R_{11}-1}{R_{12}};
\end{equation}
\item[$(iii)$]
$R_{12}\neq 0$, $\hat{R}_{12}=0$, $\hat{R}_{11}=-1$, and 
\begin{equation}
 \frac{-\hat{R}_{12}}{2}\leq \frac{R_{11}+1}{R_{12}};
\end{equation}
\item[$(iv)$] $R_{12}=\hat{R}_{12}=0$, $R_{11}=\hat{R}_{11}=1$, and $\hat{R}_{21}\leq {R}_{21};$
\item[$(v)$] $R_{12}=\hat{R}_{12}=0$, $R_{11}=\hat{R}_{11}=-1$, and ${R}_{21}\leq \hat{R}_{21}$.
\end{itemize}
\end{theorem}
\begin{remark}
We remark that in the separated setting of Theorem \ref{t2.16}, one can order all nonnegative extensions. However, this is not true in the coupled setting of Theorem \ref{t2.17}. For instance, letting $R_{11}=0=\hat{R}_{11}$ one sees that only one of the inequalities in \eqref{4.29a} can hold unless $R_{12}=\hat{R}_{12}$. Therefore there is no ordering for extensions with boundary conditions of this form unless $R_{12}=\hat{R}_{12}$.
\end{remark}
\begin{proof}[Proof of Theorem \ref{t2.17}]
Much of the proof of this result follows similarly to that of Theorem \ref{t2.16}. In fact, the analog of the decomposition \eqref{4.10a} simply becomes
\begin{equation}
T_{0,R'}=T_{0,R'}^+\oplus T_{0,R'}^-,
\end{equation}
where $T_{0,R'}^\pm:=T_{0,R'}\big|_{\ker(P\mp I)}$. One then introduces the unitary operators $U_\pm$ as before, noting that the proof of the Neumann (resp., Dirichlet) condition at the midpoint for $T_{0,R'}^+$ (resp., $T_{0,R'}^-$) follows as before once one shows how the coupled boundary conditions translate into one separated boundary condition at the endpoint $x=a$.

We begin by considering $T_{0,R'}^+$. The coupled boundary conditions in \eqref{2.26} with $\eta=0$ and $R=R'$ then give two boundary condition equations after applying the identities $\wti g(b)=\wti g(a)$ and $\wti{g}^{\, \prime}(a)=-\wti{g}^{\, \prime}(b)$ in this case:
\begin{align}
\wti g(a)&=R_{11}\wti g(a)+R_{12} \wti{g}^{\, \prime}(a), \label{4.30a}\\
-\wti{g}^{\, \prime}(a)&=R_{21}\wti g(a)+R_{11}\wti{g}^{\, \prime}(a).\label{4.31a}
\end{align}
The key observation now is that the two equations are consistent with one another. For instance, if $R'=-I_2$, then the two equations state that $\wti g(a)=0$ and $\wti{g}^{\, \prime}(a)=\wti{g}^{\, \prime}(a)$, which is simply the Friedrichs condition at $x=a$. In fact, \eqref{4.30a} shows that the Friedrichs condition holds whenever $R_{12}=0$ and $R_{11}=-1$. 
Otherwise, if $R_{21}=0$ and $R_{11}=-1$, then $\eqref{4.31a}$ is again vacuous and \eqref{4.30a} simplifies to 
\begin{equation}\label{4.32a}
-\frac{2}{R_{12}}\wti g(a)+\wti{g}^{\, \prime}(a)=0.
\end{equation}
Finally, if $R_{11}\neq-1$ and \eqref{4.31a} holds, one can solve for $\wti{g}^{\, \prime}(a)$ in \eqref{4.31a} and substitute into the right-hand side of \eqref{4.30a} to verify that \eqref{4.30a} holds since
\begin{equation}
R_{11}\wti g(a)+R_{12} \wti{g}^{\, \prime}(a)=R_{11}\wti g(a)-\frac{R_{12}R_{21}}{R_{11}+1} \wti{g}(a)=\wti g(a),
\end{equation}
where we used that $\det(R')=1$. Therefore, moving everything to one side of \eqref{4.30a} and dividing yields
\begin{equation}\label{4.34a}
\frac{R_{11}-1}{R_{12}}\wti g(a)+\wti{g}^{\, \prime}(a)=0,\quad R_{12}\neq0.
\end{equation}
Comparing \eqref{4.31a} (for $R_{11}=1,\ R_{12}=0$) and \eqref{4.34a} with \eqref{2.25} one obtains the identification for $\alpha'\in(0,\pi]$
\begin{align}
&\cot(\alpha')=\frac{R_{11}-1}{R_{12}}\ \text{ for } R_{12}\neq0,\quad \cot(\alpha')=\frac{R_{21}}{2}\ \text{ for } R_{11}=1,\; R_{12}= 0, \notag \\
\quad
&\text{and}\ \alpha'=\pi \text{ for } R_{11}=-1,\; R_{12}=0.
\end{align}
A similar analysis when considering $T_{0,R'}^-$ and using the identities $\wti g(b)=-\wti g(a)$ and $\wti{g}^{\, \prime}(a)=\wti{g}^{\, \prime}(b)$ leads to the identification for $\alpha\in(0,\pi]$
\begin{align}
&\cot(\alpha)=\frac{R_{11}+1}{R_{12}}\ \text{ for } R_{12}\neq0,\quad \cot(\alpha)=\frac{-R_{21}}{2}\ \text{ for } R_{11}=-1,\; R_{12}=0,\notag \\ 
&\text{and}\quad \alpha=\pi\ \text{ for } R_{11}=1,\; R_{12}=0.
\end{align}
Thus the proof of \eqref{eq:4.16.0} and \eqref{4.17} now follows exactly as before, with \eqref{4.16} an immediate consequence. In addition, the proof of \eqref{4.24a} is the same as before.

Finally, the ordering follows similarly to before utilizing the expression \eqref{4.17} while noting that there are now two parameters that must be considered.
\end{proof}

We now point out the special cases of the periodic and antiperiodic extensions:
\begin{align}\label{4.26}
\begin{split}
T_{Per}&=T_{0,I_2}\ \text{ is unitarily equivalent to }\ T^{(a,(b+a)/2)}_F\oplus T^{(a,(b+a)/2)}_{\pi/2,\pi/2},\\
T_{Anti}&=T_{0,-I_2}\ \text{ is unitarily equivalent to }\ T^{(a,(b+a)/2)}_{\pi,\pi/2}\oplus T^{(a,(b+a)/2)}_{\pi/2,\pi}.
\end{split}
\end{align}

The next corollary follows immediately from \eqref{eq:4.12.0} and \eqref{eq:4.16.0}.
\begin{corollary}
Assume Hypothesis \ref{h2.13} and let $R,\ \hat{R}$ be given such that $R_{11}=R_{22}$ and $\hat{R}_{11}=\hat{R}_{22}$. In addition, assume $\alpha=\hat{\alpha}'$ and $\alpha'=\hat{\alpha}$, where these boundary condition parameters are defined via \eqref{4.17} for $R,\ \hat{R}$. Then the following holds$:$
\begin{equation}\label{4.35a}
T_{\alpha,\alpha}\oplus T_{\alpha',\alpha'}\ \text{ is unitarily equivalent to }\ T_{0,R}\oplus T_{0,R'}.
\end{equation}
\end{corollary}
An important special case of \eqref{4.35a} is the following:
\begin{equation}
T_F\oplus T_{\pi/2,\pi/2}\ \text{ is unitarily equivalent to }\ T_{Per}\oplus T_{Anti}.
\end{equation}

As a final corollary of Theorems \ref{t2.16} and \ref{t2.17}, we have the following.

\begin{theorem}
Assume Hypothesis \ref{h2.13}. Then every self-adjoint extension for the full interval problem that is invariant with respect to reflection about the midpoint of the interval is unitarily equivalent to the direct sum of two self-adjoint extensions on the half interval: one with Dirichlet and one with Neumann boundary conditions at the regular point $x=(a+b)/2$, and both with some $($same or different$)$ general separated boundary condition at $x=a$.

Conversely, consider any pair of self-adjoint extensions on the half interval such that one has Dirichlet and one has Neumann boundary conditions at the regular point $x=(a+b)/2$, and both have some $($same or different$)$ general separated boundary condition at $x=a$. Then their direct sum is unitarily equivalent to a self-adjoint extension for the full interval that is invariant with respect to reflection about the midpoint of the interval. Moreover, if the separated boundary conditions imposed at $x=a$ are the same, say $\alpha$, then the full interval extension has $\alpha$-separated boundary conditions at $x=a$ and $x=b$. If the separated boundary conditions imposed at $x=a$ are different, then the full interval extension has coupled boundary conditions with $\eta=0$ and $R'$ given as in Theorem \ref{t2.17}.
\end{theorem}

\subsection{Two interval problem}

We now look to extend the results of Theorem \ref{t2.17} to the two interval setting, which can be accomplished by following \cite[Ch. 13]{Ze05}. Once again, we point out that such problems are often considered when studying point interactions in physical models (see, e.g., \cite{ADK98}).

Let us begin by considering a Sturm--Liouville expression $\tau^{-}$ on $(-a,0)$ with $a\in(0,\infty]$ satisfying Hypothesis \ref{h2.1}. We then define $\tau^{+}$ to be the expression which is simply $\tau^{-}$ reflected about $x=0$.\footnote{For ease of notation, we choose the middle point to be zero.} Then $\tau^{\pm}$ have maximal and minimal operators, $T^\pm_{max}$ and $T^\pm_{min}$, respectively, associated with them as before. The two interval problem then consists of considering the differential expression, $s$, on $(-a,0)\cup(0,a)$ which acts as $s\big|_{(-a,0)}=\tau^{-}$ and $s\big|_{(0,a)}=\tau^{+}$, whose maximal and minimal operators, $S_{max}$ and $S_{min}$, respectively, are given by
\begin{equation}
S_{max}=T_{max}^-\oplus T_{max}^+,\quad S_{min}=T_{min}^-\oplus T_{min}^+.
\end{equation}

\begin{hypothesis}\label{h4.8}
Assume the two interval Sturm--Liouville differential expression $s$ on $(-a,0)\cup(0,a)$ has coefficient functions $p,q,r$ which are symmetric about $x=0$ and satisfy Hypothesis \ref{h2.1} on $(-a,0)$ $($hence also on $(0,a)$$)$. In addition, assume the minimal operator, $S_{min}$, associated with $s$ is bounded from below.
\end{hypothesis}

We will assume Hypothesis \ref{h4.8} throughout the following discussion. Note that there are now 4 endpoints in this problem: $x=-a$, $x=a$, along with the the left and right limits at $x=0$, which we will denote by $x=0^\mp$, respectively. Assuming boundary conditions are needed at some of these endpoints, if we impose boundary conditions which only affect endpoints from the same side of $(-a,0)\cup(0,a)$ (e.g., all separated or coupling $x=-a$ with $x=0^-$ and $x=a$ with $x=0^+$), then the problem trivially becomes the direct sum of the problems on each half of the interval (in general regardless of symmetries). So we focus on extending our results to the case when coupling occurs between endpoints from different sides of $(-a,0)\cup(0,a)$.

Notice that if the problem has exactly two limit circle endpoints it becomes completely analogous to considering the one interval problem in terms of parameterizing self-adjoint extensions, replacing the previous $a$ and $b$ by these two limit circle endpoints (see \cite[Thm. 13.3.1 Case 3]{Ze05} where the boundary values can be written in terms of generalized boundary values just as the one interval case).

Since our problem was constructed to have coefficient functions which are symmetric about $x=0$, we once again restrict to those coupled boundary conditions which are invariant with respect to this symmetry. These are exactly the same as before: $\eta=0$ and $R_0\in SL(2,\bbR)$ such that $R_{11}=R_{22}$. We will denote these self-adjoint extensions of $S_{min}$ by $S_{R_0}$.

\begin{theorem}\label{t4.9}
Assume Hypothesis \ref{h4.8}, $x=\pm a$ are in the limit point case, and $x=0^\mp$ are in the limit circle nonoscillatory case. Let $S_{R_0}$ denote the extension of the minimal operator associated with $s$ satisfying coupled boundary conditions at $x=0^\mp$ with $\eta=0$ and $R_{11}=R_{22}$.
In addition, let $T^{+}_\a$ denote the self-adjoint extension of $s\big|_{(0,a)}$ with $\a$-separated boundary conditions applied at $x=0$. Then $S_{R_0}$ is unitarily equivalent to the direct sum
\begin{equation} \label{eq:4.37}
T^{+}_{\a}\oplus T^{+}_{\a'}
\end{equation}
in $L^2((0,a);rdx)\oplus L^2((0,a);rdx)$ with $\a,\a'$ given as in \eqref{4.17}.

Moreover, the statements regarding nonnegative extensions in Theorem \ref{t2.17} hold in this setting as well.
\end{theorem}
\begin{proof}
This follows from a completely analogous argument as given in the proof of the decomposition \eqref{eq:4.16.0}, where the even/odd subspaces with respect to a reflection about $x=0$ are introduced. Imposing the conditions
\begin{equation}
\wti g(0^+)=\wti g(0^-)\quad\mbox{and}\quad {\wti g}^{\,\prime}(0^+)=-{\wti g}^{\,\prime}(0^-)
\end{equation}
in the even case and
\begin{equation}
\wti g(0^+)=-\wti g(0^-)\quad\mbox{and}\quad {\wti g}^{\,\prime}(0^+)={\wti g}^{\,\prime}(0^-)
\end{equation}
for the odd case, the same arguments extend to this setting yielding \eqref{eq:4.37}. Moreover, note that no additional generalized boundary conditions at $\pm a$ are required, since $\pm a$ are assumed to be limit point.

The ordering of nonnegative extensions (if $S_{min}$ is strictly positive) follows as before by applying Remark \ref{Rem:3.3} $(iv)$.
\end{proof}

\begin{remark}\label{r4.10}
$(i)$ Analogous results hold whenever $x=\pm a$ are assumed to be in the limit circle nonoscillatory case with $x=0^\mp$ both being limit point.\\[1mm]
$(ii)$ If $x=\pm a\in\bbR$ are in the limit circle nonoscillatory case with fixed separated boundary condition at $x=\pm a$, say $\beta'\in(0,\pi]$, the analog of Theorem \ref{t4.9} follows immediately as above. In particular, denoting by $S_{\beta',R_0}$ the self-adjoint extension of $S_{min}$ with fixed $\beta'$-separated boundary conditions at $x=\pm a$ and $R_0$-coupled boundary conditions at $x=0^\mp$, the same arguments as before yield that $S_{\beta',R_0}$ is unitarily equivalent to the direct sum
\begin{equation}
T^{+}_{\a,\beta'}\oplus T^{+}_{\a',\beta'}
\end{equation}
in $L^2((0,a);rdx)\oplus L^2((0,a);rdx)$ with $\a,\a'$ given as in \eqref{4.17}. The results on nonnegative extensions (if $S_{min}$ is strictly positive) also remain the same.

Once again, one can interchange the behaviors of the endpoints by fixing a separated boundary conditions at $x=0^\mp$ and coupling $x=\pm a$ instead.
\end{remark}

Next, we impose a coupled boundary condition at $x=\pm a$ and a coupled boundary condition at $x=0^\mp$ (i.e., coupling $x=-a$ with $x=a$ and separately coupling $x=0^-$ with $x=0^+$). When applying these couplings, neither one of the conditions interacts with the other, hence, the same requirements apply as before when considering which of these are invariant with respect to reflection. We will denote by $R_a$ the coupling conditions on $x=\pm a$, and similarly, those at $x=0^\mp$ by $R_0$. We will denote this self-adjoint extension of $S_{min}$ by $S_{R_a,R_0}$ (suppressing the $\eta$ terms since both are necessarily zero). 

\begin{theorem}\label{t4.11}
Assume Hypothesis \ref{h4.8} and that all endpoints, $x=\pm a,0^\mp$ are in the limit circle nonoscillatory case. Let $S_{R_a,R_0}$ denote the extension of the minimal operator associated with $s$ satisfying $R_a$-coupled boundary conditions at $x=\pm a$ $($with equal diagonal entries$)$ and $R_0$-coupled boundary conditions at $x=0^\mp$ $($with equal diagonal entries$)$.
Then $S_{R_a,R_0}$ is unitarily equivalent to the direct sum
\begin{equation}
T^{+}_{\a,\beta}\oplus T^{+}_{\a',\beta'}
\end{equation}
in $L^2((0,a);rdx)\oplus L^2((0,a);rdx)$ with $\a,\a'$ corresponding to $R_0$ as
in \eqref{4.17} and $\beta,\beta'$ corresponding to $R_a$ in \eqref{4.17}.

Moreover, given two nonnegative extensions, $S_{R_a,R_0}$, $S_{\hatt R_a,\hatt R_0}$, one has $S_{R_a,R_0}\leq S_{\hatt R_a,\hatt R_0}$ if and only if $\a\leq \hatt{\a}$, $\b\leq \hatt \b$, $\a'\leq \hatt{\a}'$, and $\b'\leq \hatt \b'$ with each given as described above $($through \eqref{4.17}$)$.
\end{theorem}
\begin{proof}
Again, decomposing into even/odd subspaces with respect to the reflection about $x=0$, we obtain the requirements
\begin{align}
{\wti g}(0^+)=\pm{\wti g}(0^-)&\mbox{\quad and \quad}{\wti g}^{\,\prime}(0^+)=\mp{\wti g}^{\,\prime}(0^-),\\
{\wti g}(-a)=\pm{\wti g}(a)&\mbox{\quad and \quad}{\wti g}^{\,\prime}(-a)=\mp{\wti g}^{\,\prime}(a),
\end{align}
for $g$ to be an element of the even/odd subspace, respectively. Since we are only considering extensions for which the generalized boundary conditions at $0^-$ are coupled with those at $0^+$ and the generalized boundary conditions at $-a$ are only coupled with those at $a$, we can again mimic the argument as in the proof of the decomposition \eqref{eq:4.16.0} for each pair of coupled boundary points.

The ordering follows from Remark \ref{r3.5} $(iv)$.
\end{proof}

We end this section by remarking that one could also consider the case where all four endpoints are coupled together, studying those choices which are invariant under reflection. As this becomes more involved and technical than the results above, we will revisit this in an upcoming work.

\section{Examples and an application to integral inequalities}\label{s5}

We now consider two examples: the classic Bessel equation and a symmetric Bessel-type potential, the latter of which can also be understood as a perturbed symmetric Heun problem by \cite[App. A]{GPS24}. There is an enormous amount of literature on Bessel-type operators, and we refer the interested reader to \cite{GPS21,GPS24} and the references found therein.

\begin{example}[Bessel equation]
We begin by considering the Bessel expression
\begin{equation}
-\frac{d^2}{dx^2}+\big[\gamma^2-(1/4)\big](x-a)^{-2},\quad \gamma\in[0,1),\ x\in(a,d),\ -\infty<a<d<\infty.
\end{equation}
The range of $\gamma$ guarantees that the given problem is in the limit circle nonoscillatory case at both endpoints. Principal and nonprincipal solutions of $\tau y=0$ near $x=a$ are given by
\begin{align}
u_{a,\g}(0, x) &= (x-a)^{(1/2) + \g}, \quad \g \in [0,1), \; x \in (a,d),      \lb{5.3} \\
\hatt u_{a,\g}(0, x) &= \begin{cases} (2 \g)^{-1} (x-a)^{(1/2) - \g}, & \g \in (0,1), \\
(x-a)^{1/2} \ln(1/(x-a)), & \g =0, \end{cases} \quad x \in (a,d).  \lb{5.4} 
\end{align}
The generalized boundary values at $x=a$ then become 
\begin{align}
\wti g(a) &= \begin{cases} \lim_{x \downarrow a} g(x)\big/\big[(2 \g)^{-1} (x-a)^{(1/2)-\g}\big], &  
\gamma \in (0,1), \\[1mm]
\lim_{x \downarrow a} g(x)\big/\big[(x-a)^{1/2} \ln(1/(x-a))\big], & \gamma =0, 
\end{cases} \\
\wti g^{\, \prime} (a) &= \begin{cases} \lim_{x \downarrow a} \big[g(x) - \wti g(a) (2 \gamma)^{-1} (x-a)^{(1/2)-\g}\big]\big/(x-a)^{(1/2)+\g}, 
&\g \in (0,1), \\[1mm]
\lim_{x \downarrow a} \big[g(x) - \wti g(a) (x-a)^{1/2} \ln(1/(x-a))\big]\big/(x-a)^{1/2}, & \g =0.
\end{cases}
\end{align}
Furthermore, one readily verifies that
\begin{align}
\norm{u_{a,\g}}_{L^2}^2&=(2+2\g)^{-1}(b-a)^{2+2\g},\label{5.6a}\\
\langle\hatt u_{a,\g},u_{a,\g}\rangle_{L^2}&=\begin{cases}
(4\g)^{-1}(b-a)^2, & \g\in(0,1),\\
4^{-1}(b-a)^2[1+2\ln(1/(b-a))], & \g=0.
\end{cases}
\end{align}
Therefore, $\hatt v_{a,\g}$ in \eqref{B.3.8} for $x \in (a,d)$ is explicitly given by
\begin{align}\label{5.8a}
\hatt v_{a,\g}(0,x)&=\begin{cases}
(2\g)^{-1}(x-a)^{(1/2) - \g}-\dfrac{1+\g}{2\g(b-a)^{2\g}}(x-a)^{(1/2) + \g},& \g\in(0,1),\\[2mm]
(x-a)^{1/2} \left(\ln(1/(x-a))-\dfrac{1}{2}+\ln(b-a)\right),& \g=0,
\end{cases}
\end{align}
and one notes the general formula
\begin{equation}\label{5.9a}
\norm{\hatt v_{a,\g}}^2_{L^2}=\norm{\hatt u_{a,\g}}^2_{L^2}+\big(\langle\hatt u_{a,\g},u_{a,\g}\rangle_{L^2}-2\big)\frac{\langle\hatt u_{a,\g},u_{a,\g}\rangle_{L^2}}{\norm{u_{a,\g}}_{L^2}^2}.
\end{equation}

Now, since $x=d$ is a regular point for this problem, the generalized boundary values are simply the function and derivative value at $d$. Hence, the values of \eqref{5.3}, \eqref{5.4}, \eqref{5.8a}, and their derivatives at $x=d$ along with the values in \eqref{5.6a} and \eqref{5.9a} yield all the required values to completely parameterize all nonnegative self-adjoint extensions via Section \ref{subgen}, illustrating the ease of finding the needed values.
\end{example}

\begin{example}[Symmetric Bessel equation]\label{ex3.4} 
For this example, let $d_{(a,b)}(x)$ represent the distance to the boundary function for $a,b\in\bbR,\ a<b,$ defined by
\begin{equation}
d_{(a,b)}(x)=\begin{cases}
x-a,& x\in(a,(b+a)/2],\\
b-x,& x\in[(b+a)/2,b),
\end{cases}
\end{equation}
and consider the Sturm--Liouville expression
\begin{align}
\tau_\g=-\dfrac{d^2}{dx^2}+\big[\gamma^2-(1/4)\big]d^{-2}_{(a,b)}(x),\quad
\gamma\in[0,1),\, x\in(a,b).
\end{align}
That is, we consider the classic Bessel expression on the two halves of the interval giving a symmetric confining potential with Bessel-type singularities at both endpoints. The range of $\gamma$ guarantees that the given problem is in the limit circle nonoscillatory case at both endpoints so that we can apply our previous results.

To apply the results of Sections \ref{subnonneg} and \ref{s4}, it suffices to focus our analysis around one of the singular endpoints, which we choose to be $x=a$. As such, the principal and nonprincipal solutions of $\tau y=0$ near $x=a$ on the first half of the interval are simply given by \eqref{5.3} and \eqref{5.4} with $d=(b+a)/2$, and the generalized boundary values at $x=a$ are the same as the previous example.

We now explicitly give the spectrum of the Friedrichs extension for this problem via \eqref{4.15}. Given the generalized boundary values, one readily verifies that the principal behaving solution $($see Appendix \ref{App}$)$ on $(a,(b+a)/2]$ is given by
\begin{equation}
\varphi_\g(z,x,a)=2^\g \Gamma(1+\g) z^{- \g/2}(x-a)^{1/2}J_\g\big(z^{1/2}(x-a)\big),\quad x\in(a,(b+a)/2],\; z\in\bbC,
\end{equation}
where $J_{\g}(\dott)$ is the standard Bessel functions of order $\g \in \bbR$ 
$($cf.\ \cite[Ch.~10]{DLMF}$)$.

The eigenvalues for the Friedrichs extension, $T_{F,\g}$, of the problem on $(a,b)$ will be given by the union of the Friedrichs and mixed Dirichlet--Neumann eigenvalues from the half-interval problem by Theorem \ref{t2.16}. By \eqref{2.43} on the half-interval, these eigenvalues are given, respectively, by the zeros $($w.r.t. $z\in\bbC$$)$ of 
\begin{equation}
z^{- \g/2}J_\g\big(z^{1/2}(b-a)/2\big),
\end{equation}
and the zeros $($w.r.t. $z\in\bbC$$)$ of
\begin{equation}\label{5.9}
z^{- \g/2}\Big[(x-a)^{1/2}J_\g\big(z^{1/2}(x-a)\big)\Big]'\Big|_{x=\frac{b+a}{2}}.
\end{equation}
Expanding \eqref{5.9} and letting $y=z^{1/2}(b-a)/2$ shows $z$ being a zero of \eqref{5.9} is equivalent to $z=4\lambda^2_{\g,k}/(b-a)^2$ where $\{\lambda_{\g,k}\}_{k\in\bbN}$ are the positive zeros of
\begin{align}\label{5.15}
G_{\g}(y)&=y^{-\g}[(1-2\g) J_\g(y)+2y J_{\g-1}(y)]  \notag\\
&=y^{-\g}\bigg( J_\g(y)+2y \frac{d}{dy}[J_\g(y)]\bigg),\quad \g\in[0,1).
\end{align}
We remark that the first positive zero, $\lambda_{\g,1}$, is sometimes referred to as Lamb's constant $($see the brief discussion in \cite{AW07}$)$.
One concludes that
\begin{equation}
\sigma(T_{F,\g})=\sigma(T^{(a,(b+a)/2)}_{F,\g})\cup\sigma(T^{(a,(b+a)/2)}_{D-N,\g}),\quad \gamma\in[0,1),
\end{equation}
where
\begin{align}\label{5.13}
\sigma(T^{(a,(b+a)/2)}_{F,\g})&=\bigcup_{k\in\bbN}\bigg\{ \frac{4j^2_{\gamma,k}}{(b-a)^2}\bigg\},\quad
\sigma(T^{(a,(b+a)/2)}_{D-N,\g})=\bigcup_{k\in\bbN}\bigg\{ \frac{4\lambda^2_{\g,k}}{(b-a)^2}\bigg\},
\end{align}
with $\{j_{\g,k}\}_{k\in\bbN}$ denoting the positive zeros of the Bessel function $J_\g(\dott)$.

Since $T_{min,\g}$ is positive for this example, we can now apply the previous results to characterize the nonnegative extensions for this problem. In particular, applying \eqref{4.14} shows that the $(\a,\a)$-extension is nonnegative for $\alpha\in[\nu,\pi]$ where
\begin{equation}\label{5.14}
\nu=\begin{cases}
\cot^{-1}\big((2\g)^{-1}[(1-2\g)/(1+2\g)][2/(b-a)]^{2\g}\big),& \g\in(0,1),\\[1mm]
\cot^{-1}\big(\ln(2/(b-a))-2\big),& \g=0.
\end{cases}
\end{equation}
For symmetric coupled boundary conditions, we refer to \eqref{4.17} and \eqref{4.24a} noting that in this case $\mu$ is defined as in \eqref{5.14} and
\begin{equation}
\nu=\begin{cases}
\cot^{-1}\big((2\g)^{-1}[2/(b-a)]^{2\g}\big),& \g\in(0,1),\\[1mm]
\cot^{-1}\big(\ln(2/(b-a))\big),& \g=0.
\end{cases}
\end{equation}
\end{example}

\subsection{An application to integral inequalities}

We finish with a brief remark regarding applying these results to integral inequalities. Example \ref{ex3.4} allows one to prove (following the proof of \cite[Thm. 3.1]{GPS21} for instance) the following integral inequality that was originally proven in \cite{AW07}:
\begin{align}\label{5.16}
\int_a^b  |f'(x)|^2\, dx  \geq \big[(1/4)-\g^2\big]\int_a^b d_{(a,b)}^{-2}(x)|f(x)|^2\, dx
+\frac{4\lambda_{\g,1}^2}{(b-a)^2}\int_a^b & |f(x)|^2\, dx,  \notag \\
\g\in[0,1),\ f  \in H_0^1((a,b)),\ a,b\in\bbR&,
\end{align}
where $H_0^1 ((a,b))$ denotes the standard Sobolev space on $(a,b)$ obtained upon completion of $C_0^\infty ((a,b))$ in the norm of $H^1 ((a,b))$.

In fact, one concludes from the Theorem \ref{t2.16} that the optimal constant on the integral of $|f|^2$ for an integral inequality (on an appropriate function space such as the Friedrichs domain) related to a Sturm--Liouville problem with fixed symmetric coefficient functions will simply be the lowest eigenvalue of the mixed Dirichlet--Neumann extension of the corresponding half-interval problem.
In particular, this illustrates why the multiple of Lamb's constant, $\lambda_{\g,1}$, naturally appears in the integral inequality \eqref{5.16} since it is the smallest eigenvalue of the half-interval Dirichlet--Neumann extension by \eqref{5.13}.

\appendix

\section{Symmetries of the fundamental system of solutions}\label{App}

This appendix includes a discussion
of the fundamental system of solutions and a direct factorization of
the characteristic function associated with symmetric coefficient problems. We begin by introducing the fundamental system of solutions $\theta(z,x,a)$, $\varphi(z,x,a)$ of the nonoscillatory quasi-regular problem $\tau y=z y$ defined by
\begin{align}
\wti\theta(z,a,a)=\wti\varphi^{\, \prime}(z,a,a)=1,\quad \wti\theta^{\, \prime}(z,a,a)=\wti\varphi(z,a,a)=0,
\end{align}
such that $W\big(\theta(z,\dott,a),\varphi(z,\dott,a)\big)=1,$
where $\wti{g}(\dott)$ is given as in Theorem \ref{t2.10}. Notice that this implies that the solution $\varphi$ is principal at $x=a$ whereas the solution $\theta$ is nonprincipal at $x=a$.

For later use, we next introduce characteristic functions (see \cite{FPS24} for more details) by means of the generalized boundary values for $g,g^{[1]}\in\ACl$,
\begin{align}
U_{\a}(g)&=\wti g(a)\cos(\al)+\wti g^{\, \prime}(a)\sin(\al),\\
U_{\b}(g)&=\wti g(b)\cos(\be)-\wti g^{\, \prime}(b)\sin(\be),
\end{align}
in the case ($i$) of separated boundary conditions in Theorem \ref{t2.10}, and
\begin{align}
V_{\eta,R,1}(g)&=\wti g(b)-e^{i\eta}R_{11}\wti g(a)-e^{i\eta}R_{12}\wti g^{\, \prime}(a),\\
V_{\eta,R,2}(g)&=\wti g^{\, \prime}(b)-e^{i\eta}R_{21}\wti g(a)-e^{i\eta}R_{22}\wti g^{\, \prime}(a),
\end{align}
in the case ($ii$) of coupled boundary conditions in Theorem \ref{t2.10}. We can then define the \emph{characteristic functions}
\begin{align}
F_{\a,\b}(z)&=\det\begin{pmatrix}U_{\a}(\theta(z,\dott,a))& U_{\a}(\varphi(z,\dott,a))\\ U_{\b}(\theta(z,\dott,a)) & U_{\b}(\varphi(z,\dott,a))\end{pmatrix},\\
F_{\eta,R}(z)&=\det\begin{pmatrix}V_{\eta,R,1}(\theta(z,\dott,a))& V_{\eta,R,1}(\varphi(z,\dott,a))\\ V_{\eta,R,2}(\theta(z,\dott,a)) & V_{\eta,R,2}(\varphi(z,\dott,a))\end{pmatrix}.
\end{align}

By construction, eigenvalues or $T_{A,B}$ are determined via $F_{A,B}(z)=0$, with multiplicity of eigenvalues of $T_{A,B}$ corresponding to multiplicity of zeros of $F_{A,B}$, and $F_{A,B}(z)$ is entire with respect to $z$ (where $A,B$ represents $\alpha,\beta$ in the case of separated boundary conditions and $\eta,R$ in the context of coupled boundary conditions). In particular, for $T_{\a,\b}$, one has 
\begin{align}\label{2.43}
\begin{split}
F_{\a,\b}(z)&=\cos(\a)[-\sin(\b)\ \wti\varphi^{\, \prime}(z,b,a)+\cos(\b)\ \wti\varphi(z,b,a)] \\
&\quad\,-\sin(\a)[-\sin(\b)\ \wti\theta^{\, \prime}(z,b,a)+\cos(\b)\ \wti\theta(z,b,a)],\quad \a,\b\in(0,\pi],
\end{split}
\end{align}
and for $T_{\eta,R}$, that is, coupled boundary conditions, one has
\begin{align}\label{2.44}
\notag F_{\eta,R}(z)&=e^{i\eta}\big(R_{12}\wti\theta^{\, \prime}(z,b,a)-R_{22}\wti\theta(z,b,a)+R_{21}\wti\varphi(z,b,a)-R_{11}\wti\varphi^{\, \prime}(z,b,a)\big)\\
&\quad\,+e^{2i\eta}+1,\quad \eta\in[0,\pi),\ R\in SL(2,\R).
\end{align}

Assuming Hypothesis \ref{h2.13}, now consider the two fundamental systems of solutions, $\{\varphi_a,\theta_a\}$ for $\tau\big|_{(a,(b+a)/2)} y=zy$ and $\{\varphi_b,\theta_b\}$ for $\tau\big|_{((b+a)/2,b)} y=zy$ such that the solutions are reflections of one another about the line $x=(b+a)/2$, that is, 
\begin{align}
\begin{split}
\varphi_a(z,x)&=\varphi_b(z,b-x+a)\\
\theta_a(z,x)&=\theta_b(z,b-x+a),\quad x\in(a,(b+a)/2],
\end{split}
\end{align}
with
\begin{align}
\begin{split}\label{4.5}
\wti \varphi_a(z,a)=\wti \theta'_a(z,a)=0=\wti \varphi_b(z,b)=-\wti \theta_b'(z,b),\\
\wti \varphi'_a(z,a)=\wti \theta_a(z,a)=1=-\wti \varphi'_b(z,b)=\wti \theta_b(z,b),
\end{split}
\end{align}
where the generalized boundary values at $x=b$ are defined by using the reflection of the principal and nonprincipal solutions at $x=a$.
Note that these assumptions imply $W(\theta_a,\varphi_a)=1=-W(\theta_b,\varphi_b)$ and
\begin{align}
\begin{split}
\varphi^{[1]}_a(z,x)&=-\varphi_b^{[1]}(z,b-x+a)\\
\theta^{[1]}_a(z,x)&=-\theta_b^{[1]}(z,b-x+a),\quad x\in(a,(b+a)/2].
\end{split}
\end{align}

Now to construct the fundamental system of solutions to $\tau y=zy$ on the entire interval $(a,b)$, one can utilize the fundamental systems of solutions on either half of the interval to define
\begin{equation}\label{4.8}
\varphi(z,x,a)=\begin{cases}
\varphi_a(z,x),& x\in(a,(b+a)/2],\\
\big(1-2\theta_a(z,(b+a)/2)\varphi_a^{[1]}(z,(b+a)/2)\big)\varphi_b(z,x)\\
\quad+2\varphi_a(z,(b+a)/2)\varphi_a^{[1]}(z,(b+a)/2)\theta_b(z,x),& x\in[(b+a)/2,b),
\end{cases}
\end{equation}
\begin{equation}\label{4.9}
\theta(z,x,a)=\begin{cases}
\theta_a(z,x),& x\in(a,(b+a)/2],\\
\big(1+2\varphi_a(z,(b+a)/2)\theta_a^{[1]}(z,(b+a)/2)\big)\theta_b(z,x)\\
\quad -2\theta_a(z,(b+a)/2)\theta_a^{[1]}(z,(b+a)/2)\varphi_b(z,x),& x\in[(b+a)/2,b).
\end{cases}
\end{equation}
One readily verifies that both functions satisfy $\tau y=zy$, they and their quasi-derivative are absolutely continuous locally on $(a,b)$, and $W(\theta,\varphi)=1$. We also remark that the coefficient on $\varphi_b$ in \eqref{4.8} (which is $-\wti{\varphi}^{\, \prime}(z,b,a)$) can be written
\begin{equation}\label{4.10}
1-2\theta_a(z,(b+a)/2)\varphi_a^{[1]}(z,(b+a)/2)=W(\theta_a(z,\,\cdot\,),\varphi_b(z,\,\cdot\,))\big|_{x=(b+a)/2}.
\end{equation}
One can also verify $\wti{\varphi}^{\, \prime}(z,b,a)$ and $\wti{\theta}(z,b,a)$ are equal by applying $W(\theta_a,\varphi_a)=1$ as
\begin{align}\label{4.11}
&\wti{\theta}(z,b,a)-\wti{\varphi}^{\, \prime}(z,b,a)=\\
&1+2\varphi_a(z,(b+a)/2)\theta_a^{[1]}(z,(b+a)/2)+1-2\theta_a(z,(b+a)/2)\varphi_a^{[1]}(z,(b+a)/2)=0.\notag
\end{align}
This illustrates how the form of the solutions given in \eqref{4.8} and \eqref{4.9} allows one to easily compute the boundary values at the endpoint $x=b$ in terms of combinations of values at the midpoint for the function normalized at $x=a$.

\subsection{Factorizations of the characteristic functions} In this subsection, we illustrate how one can factorize the characteristic functions associated to the symmetric problems considered in Section \ref{s4}.

The separated case follows immediately from setting $\beta=\alpha$ in \eqref{2.43} and applying generalized boundary conditions to \eqref{4.8} and \eqref{4.9} to write (suppressing the arguments $(z,(b+a)/2)$ for $\varphi_a,\varphi_a^{[1]},\theta_a,\theta_a^{[1]}$ for brevity)
\begin{align}
F_{\alpha,\alpha}(z)&=\cos(\a)[-\sin(\a)(2\theta_a\varphi_a^{[1]}-1)+2\cos(\a)\varphi_a\varphi_a^{[1]}] \notag\\
&\quad\,-\sin(\a)[-2\sin(\a)\theta_a \theta_a^{[1]}+\cos(\a)(1+2\varphi_a \theta_a^{[1]})]  \notag\\
&=-2(\cos(\alpha)\varphi_a-\sin(\alpha)\theta_a)(\sin(\alpha)\theta_a^{[1]}-\cos(\alpha)\varphi_a^{[1]} )  \notag \\
&=-2F^{(a,(b+a)/2)}_{\alpha,\pi}(z)F^{(a,(b+a)/2)}_{\alpha,\pi/2}(z).
\end{align}
Here, $F^{(a,(b+a)/2)}_{A,B}(z)$ denotes the characteristic function given by \eqref{2.43} when considering $\tau\big|_{(a,(b+a)/2)}$ with regular boundary conditions at $x=(b+a)/2$. 

For the coupled case, we begin with $(\eta,R)=(0,-I_2)$ (i.e., the antiperiodic case). Using that $\wti{\varphi}^{\, \prime}(z,b,a)=\wti{\theta}(z,b,a)$ (see \eqref{4.11}), we find from \eqref{2.44}, \eqref{4.8}, and \eqref{4.9} with antiperiodic boundary conditions,
\begin{align}
F_{0,-I_2}(z)&= 2+\big(\wti{\theta}(z,b,a)+\wti{\varphi}^{\, \prime}(z,b,a)\big) \\
&=2\theta_a(z,(b+a)/2)\varphi_a^{[1]}(z,(b+a)/2)=-2F^{(a,(b+a)/2)}_{\pi/2,\pi}(z)F^{(a,(b+a)/2)}_{\pi,\pi/2}(z). \notag
\end{align}
The other cases with $R_{11}=-1$ but only one of $R_{12}$ or $R_{21}$ being zero, or $R_{11}=1$, follow similarly, so we omit them here.

We next treat the case whenever $R_{11}\neq\pm1$. From \eqref{2.44}, \eqref{4.8}, and \eqref{4.9}, for general coupled boundary conditions satisfying $\eta=0,\ R_{11}=R_{22}$, one uses that $\wti{\varphi}^{\, \prime}(z,b,a)=\wti{\theta}(z,b,a)$ to write (suppressing the arguments $(z,(b+a)/2)$ for $\varphi_a,\varphi_a^{[1]},\theta_a,\theta_a^{[1]}$ for brevity)
\begin{align}\label{4.22}
F_{0,R'}(z)&=R_{12}\wti\theta^{\, \prime}(z,b,a)+R_{21}\wti\varphi(z,b,a)-2R_{11}\wti\theta(z,b,a)+2  \notag\\
&=2R_{12}\theta_a\theta_a^{[1]}+2R_{21}\varphi_a\varphi_a^{[1]}-4R_{11}\varphi_a\theta_a^{[1]}-2R_{11}+2 \\
&=-2\big[(R_{11}+1)\varphi_a\theta_a^{[1]}+(R_{11}-1)\theta_a\varphi_a^{[1]}-R_{12}\theta_a\theta_a^{[1]}-R_{21}\varphi_a\varphi_a^{[1]}\big]. \notag
\end{align}
We used that $W(\theta_a,\varphi_a)=1$ to introduce appropriate terms to arrive at the last equality in \eqref{4.22}. 
Next, one readily verifies from \eqref{2.43} for $\alpha,\alpha'\in(0,\pi)$,
\begin{align}\label{4.23}
&F^{(a,(b+a)/2)}_{\alpha,\pi}(z)F^{(a,(b+a)/2)}_{\alpha',\pi/2}(z)  \\
&\quad =\sin(\alpha)\sin(\alpha')\big[\cot(\alpha)\varphi_a\theta_a^{[1]}+\cot(\alpha')\theta_a\varphi_a^{[1]}-\theta_a\theta_a^{[1]}-\cot(\alpha)\cot(\alpha')\varphi_a\varphi_a^{[1]}\big].\notag
\end{align}
Now suppose that $R_{11}\neq\pm1$ and compare \eqref{4.22} to \eqref{4.23} in order to write
\begin{align}
F_{0,R'}(z)&=-2(R_{11}^2-1)\bigg[\frac{1}{R_{11}-1}\varphi_a\theta_a^{[1]}+\frac{1}{R_{11}+1}\theta_a\varphi_a^{[1]}-\frac{R_{12}}{R_{11}^2-1}\theta_a\theta_a^{[1]}  \notag\\
&\hspace{3cm}-\frac{R_{21}}{R_{11}^2-1}\varphi_a\varphi_a^{[1]}\bigg] \\
&=-2\frac{R_{12}}{\sin(\alpha)\sin(\alpha')}F^{(a,(b+a)/2)}_{\alpha,\pi}(z)F^{(a,(b+a)/2)}_{\alpha',\pi/2}(z),\ R_{11}\neq\pm1,\, \alpha,\alpha'\in(0,\pi),\notag
\end{align}
by identifying for $R_{11}\neq\pm1$, $\cot(\alpha)=R_{21}/(R_{11}-1)$ and $\cot(\alpha')=R_{21}(R_{11}+1)$.

\subsection{The Krein--von Neumann extension} We end by considering the following important extension, which is exactly the Krein--von Neumann extension $T_{0,R_K}$ of $T_{min}$ whenever $T_{min}$ is strictly positive (i.e., $T_{min} \geq \varepsilon I$ for some 
$\varepsilon > 0$) by Theorem \ref{t2.12} $(ii)$. It is also the unique self-adjoint extension that has zero as an eigenvalue with multiplicity two. In particular, we consider
\begin{align}
\begin{split} 
& T_{0,R'} f = \tau f,    \\
& f \in \dom(T_{0,R'})=\bigg\{g\in\dom(T_{max}) \, \bigg| \begin{pmatrix} \wti g(b) 
\\ {\wti g}^{\, \prime}(b) \end{pmatrix} = R' \begin{pmatrix}
\wti g(a) \\ {\wti g}^{\, \prime}(a) \end{pmatrix} \bigg\},  
\end{split}
\end{align}
where, by applying \eqref{4.8}, \eqref{4.9}, and \eqref{4.11},
\begin{align}\label{4.29}
&R'=\begin{pmatrix} \wti{\theta}(0,b,a) & \wti \varphi(0,b,a)\\
\wti{\theta}^{\, \prime}(0,b,a) & \wti \varphi^{\, \prime}(0,b,a)
\end{pmatrix}=\begin{pmatrix} \wti{\theta}(0,b,a) & \wti \varphi(0,b,a)\\
\wti{\theta}^{\, \prime}(0,b,a) & \wti \theta(0,b,a)
\end{pmatrix}\\
&\hspace{-2pt}=\begin{pmatrix} 1+2\varphi_a(0,(b+a)/2)\theta_a^{[1]}(0,(b+a)/2) &  2\varphi_a(0,(b+a)/2)\varphi_a^{[1]}(0,(b+a)/2)\\
2\theta_a(0,(b+a)/2)\theta_a^{[1]}(0,(b+a)/2) & 2\theta_a(0,(b+a)/2)\varphi_a^{[1]}(0,(b+a)/2)-1
\end{pmatrix}\hspace{-3pt}.\notag
\end{align}

Hence when assuming Hypothesis \ref{h2.13}, the extension defined by \eqref{4.29} always satisfies the characterization in Theorem \ref{t2.17} and substituting the entries of \eqref{4.29} into \eqref{4.17} yields the explicit decomposition into separated boundary condition extensions on the half interval,
$T^{(a,(b+a)/2)}_{\alpha,\pi}$ and $T^{(a,(b+a)/2)}_{\alpha',\pi/2}$, where
\begin{equation}
\cot(\alpha)=\frac{\theta_a(0,(b+a)/2)}{\varphi_a(0,(b+a)/2)},\quad
\cot(\alpha')=\frac{\theta_a^{[1]}(0,(b+a)/2)}{\varphi_a^{[1]}(0,(b+a)/2)},\quad \a,\a'\in(0,\pi).
\end{equation}

\medskip

\noindent {\bf Acknowledgments.}
We thank Fritz Gesztesy for valuable feedback and the anonymous referee for helpful suggestions. J.S. appreciates the hospitality of the Baylor University math department where part of this work was completed.


\end{document}